\documentclass[10pt,reqno]{amsart}
\usepackage{epic}
\usepackage{arydshln}
\usepackage{color}
\usepackage{tikz}
\usepackage{enumitem}
\usepackage{amsmath}
\usepackage{amssymb}
\usepackage{amsthm}

\usepackage{breqn}

\newtheorem{theorem}{Theorem}[section]
\newtheorem{lemma}[theorem]{Lemma}
\newtheorem{corollary}[theorem]{Corollary}

\newtheorem{conjecture}[theorem]{Conjecture}

\newtheorem{remark}[]{Remark}

\title{Andrews-Beck type Congruences related to the Crank of a Partition}
\author{Shreejit Bandyopadhyay}

\begin{document}

\maketitle

\footnotetext[1]{Penn State University, email-sxb437@psu.edu}

\begin{abstract}
    In this paper, we discuss a few recent conjectures made by George Beck related to the ranks and cranks of partitions. The conjectures for the rank of a partition were proved by Andrews by using results due to Atkin and Swinnerton-Dyer on a suitable generating function, while the conjectures related to cranks were studied by Shane Chern using weighted partition moments. We revisit the conjectures on the crank of a partition by decomposing the relevant generating function and further explore connections with Apple-Lerch series and tenth order mock theta functions.
\end{abstract}

\section{Introduction}
The partition of a positive integer $n$ is a weakly decreasing sequence of integers whose sum equals $n$, and is usually denoted by $p(n).$ For example, the number 4 has 5 partitions: 4, 3+1, 2+2, 2+1+1, 1+1+1+1 and as such, $p(4)=5.$ \\ \\
Among the identities satisfied by the partition function $p(n),$ we have the three celebrated congruences due to Ramanujan:
\begin{equation} 
p(5n+4)\equiv 0\pmod 5
\end{equation}
\begin{equation}
p(7n+5)\equiv 0\pmod 7
\end{equation}
\begin{equation}
p(11n+6)\equiv 0\pmod {11}
\end{equation} \\
Dyson defined the rank of a partition to be its largest part minus the number of parts and also conjectured the existence of a crank function that should combinatorially explain all of the above three congruences of Ramanujan. The actual crank function was discovered much later by Andrews and Garvan and is defined as follows. For a partition $\lambda$ of $n,$ if $l(\lambda)$ is the largest part, $\omega(\lambda)$ the number of ones and $\mu(\lambda)$ the number of parts in $\lambda$ bigger than $\omega(\lambda),$ then the crank of $\lambda$ is given by \\  \par\par $c(\lambda)=$
$\begin{cases}
l(\lambda) &\mbox{if } \omega(\lambda)=0 \\
\mu(\lambda)-\omega(\lambda) &\mbox{if } \omega(\lambda)>0
\end{cases}$\\ \\
In their paper [1], Atkin and Swinnerton-Dyer proved (1) and (2) by defining $N(m,k,n)$ to be the number of partitions of $n$ with rank congruent to $m$ modulo $k$ and then establishing that $N(i,5,5n+4)=\displaystyle\frac{1}{5}p(5n+4)$ for $0\leq i\leq 4$ and $N(i,7,7n+5)=\displaystyle\frac{1}{7}p(7n+5)$ for $0\leq i\leq 6.$ \\ \\
Instead of the function $N(m,k,n)$, the conjectures made by George Beck considered the number of parts in the partitions of $n$ with rank congruent to $m$ modulo $k,$ denoted by $NT(m,k,n)$ and the number of ones in the partitions of $n$ with crank congruent to $m$ modulo $k,$ denoted by $M_{\omega}(m,k,n).$ We now make a note of these results. 
\begin{theorem}
If $i=1,4$, then for $\forall n\ge 0,$ $NT(1,5,5n+i)-NT(4,5,5n+i)+2NT(2,5,5n+i)-2NT(3,5,5n+i) \equiv 0\pmod 5.$
\end{theorem}
\begin{theorem}
If $i=1,5$, then for $\forall n\ge 0,$ $NT(1,7,7n+i)-NT(6,7,7n+i)+NT(2,7,7n+i)-NT(5,7,7n+i)-NT(3,7,7n+i)+NT(4,7,7n+i)\equiv 0\pmod 7.$
\end{theorem}
\begin{theorem}
For $\forall n\ge 0,$ $M_{\omega}(1,5,5n+4)+2M_{\omega}(2,5,5n+4)-2M_{\omega}(3,5,5n+4)-M_{\omega}(4,5,5n+4)\equiv 0\pmod 5.$
\end{theorem}
\begin{theorem}
For $\forall n\ge 0,$ $M_{\omega}(1,7,7n+5)+2M_{\omega}(2,7,7n+5)+3M_{\omega}(3,7,7n+5)-3M_{\omega}(4,7,7n+5)-2M_{\omega}(5,7,7n+5)-M_{\omega}(6,7,7n+5)\equiv 0\pmod 7.$
\end{theorem}
\begin{remark} 
As observed by Andrews in [2], Theorems 1.1 and 1.2 are trivial if we replace the function $NT$ by the function $N$ of Atkin and Swinnerton-Dyer since the rank function is symmetric: $N(m,k,n)=N(k-m,k,n).$ 
\end{remark}
For the rest of this paper, we adopt the standard q-series notations:\\ 
\par $(A;q)_n=\displaystyle\prod_{m=0}^{n-1}(1-Aq^m)$,  $(A;q)_{\infty}=\displaystyle\prod_{m=0}^{\infty}(1-Aq^m),(A_1,A_2,...A_r;q)_n=\displaystyle\prod_{t=1}^{r}(A_t;q)_n,\\ \\(A_1,A_2,...A_r;q)_{\infty}=\displaystyle\prod_{t=1}^{r}(A_t;q)_{\infty}.$
\section{Preliminaries}
Theorems 1.1 and 1.2 were proved by Andrews in [2] by considering the generating function which weights partitions by the number of parts while keeping track of the rank. That function is 
\begin{equation}
    \frac{\partial}{\partial x}\Bigr\rvert_{x = 1}\displaystyle\sum_{n\geq 0}\displaystyle\frac{x^nq^{n^2}}{(zq;q)_n(\displaystyle\frac{xq}{z};q)_n}
\end{equation}
where $q$ marks the number being partitioned, $z$ marks the rank and $x$ keeps track of the number of parts in the partition. As noted by Andrews in [2], this is a more complicated function compared to the universal mock-theta function $\displaystyle\sum_{n\geq 0}\displaystyle\frac{q^{n^2}}{(zq;q)_n(\displaystyle\frac{q}{z};q)_n}$ which counts the number of partitions with a certain rank, with $q$ marking the number being partitioned and $z$ counting the rank itself.\\ \\
Theorems 1.1-1.4 were independently considered by Shane Chern in [3] by relating them to the second rank and crank moments $N_2(n)$ and $M_2(n)$ respectively. \\ \\
In this paper, our goal will be to deal with Theorems 1.3 and 1.4 by considering, instead of the function in (4), the following function
\begin{equation}
    \displaystyle\frac{(xq)_{\infty}}{(zq,\displaystyle\frac{xq}{z};q)_{\infty}}
\end{equation}
where the variable $x$ keeps track of the number of ones in a partition, $q$ marks the number being partitioned and $z$ marks the crank of the partition. This function is somewhat simpler than the function in (4) used by Andrews in [2] to deal with Theorems 1.1 and 1.2. \\ \\
Our goal will be to decompose the function in (5) into powers of $q$ modulo 5 or modulo 7 and thereby reduce the proofs of Theorems 1.3 and 1.4 into proving certain equalities we obtain on comparing coefficients for different powers of $q.$
\section{Generating Functions}
We will use the following theorem to decompose the crank generating function in (5) into powers of $q$ modulo 5 or 7. 
\begin{theorem}
$\displaystyle\frac{(xq)_{\infty}}{(zq,\displaystyle\frac{xq}{z};q)_{\infty}}=\displaystyle\frac{1}{(q;q)_{\infty}}\bigg[1+\displaystyle\sum_{n\geq 1}(-1)^{n-1}\displaystyle\frac{(xq;q)_n}{(q;q)_{n-1}}q^{\binom{n+1}{2}}\bigg(\displaystyle\frac{1}{q^n(1-zq^n)}+\displaystyle\frac{x/z}{(1-\displaystyle\frac{xq^n}{z})}\bigg)\bigg].$
\end{theorem}
\begin{remark}
We note here the similarity of this result and Theorem 3 in [2].
\end{remark}
\begin{corollary}
If $M_{\omega}(b,k,n)$ is the number of ones in the partitions $\lambda$ of $n$ with crank congruent to $b$ modulo $k$, then 
\begin{equation*}
    \displaystyle\sum_{n\geq 0}M_{\omega}(b,k,n)q^n=\frac{\partial}{\partial x}\Bigr\rvert_{x = 1}\displaystyle\frac{1}{(q;q)_{\infty}}\displaystyle\sum_{n\geq 1}(-1)^{n-1}\displaystyle\frac{(xq;q)_n}{(q;q)_{n-1}}q^{\binom{n+1}{2}}\bigg(\displaystyle\frac{q^{n(b-1)}}{1-q^{nk}}+\displaystyle\frac{x^{k-b}q^{(k-b-1)n}}{1-x^kq^{kn}}\bigg).
\end{equation*}
\end{corollary}
\begin{corollary}
\begin{equation*}
\begin{split}
\displaystyle\sum_{n\geq 0}\bigg[M_{\omega}(1,5,n)+2M_{\omega}(2,5,n)-2M_{\omega}(3,5,n)-M_{\omega}(4,5,n)\bigg]q^n\equiv \\ \displaystyle\frac{1}{(q;q)_{\infty}}\displaystyle\sum_{n\geq 1}(-1)^nq^{\binom{n+1}{2}}\displaystyle\frac{(1-q^n)^3(1+q^n)}{(1-q^{5n})}\pmod 5.
\end{split}
\end{equation*}
\end{corollary}
\begin{corollary}
\begin{equation*}
\begin{split}
\displaystyle\sum_{n\geq 0}\bigg[M_{\omega}(1,7,n)+2M_{\omega}(2,7,n)-+3M_{\omega}(3,7,n)-3M_{\omega}(4,7,n)-2M_{\omega}(5,7,n)-M_{\omega}(6,7,n)\bigg]q^n\equiv \\ \displaystyle\frac{1}{(q;q)_{\infty}}\displaystyle\sum_{n\geq 1}(-1)^nq^{\binom{n+1}{2}}\displaystyle\frac{(1-q^n)^5(1+q^n)}{(1-q^{7n})}\pmod 7.
\end{split}
\end{equation*}
\end{corollary}
\vspace{1cm}
\begin{proof}[Proof of Theorem 3.1]
We have the following identity
\begin{equation}
    ^6\phi_5\Bigg[\begin{array}{c}
             a, q\sqrt a,-q\sqrt a,b,c,d \\
             \sqrt a,-\sqrt a,\displaystyle\frac{qa}{b},\displaystyle\frac{qa}{c},\displaystyle\frac{qa}{d}
    \end{array}
;q,\displaystyle\frac{qa}{bcd}\Bigg]=\displaystyle\frac{(aq,\displaystyle\frac{aq}{bc},\displaystyle\frac{aq}{bd},\displaystyle\frac{aq}{cd};q)_{\infty}}{(\displaystyle\frac{aq}{b},\displaystyle\frac{aq}{c},\displaystyle\frac{aq}{d},\displaystyle\frac{aq}{bcd};q)_{\infty}}
\end{equation}
Set $a=x,b=\displaystyle\frac{x}{z},c=z,d\rightarrow\infty$ in (6) whence the right side reduces to $\displaystyle\frac{(xq,q;q)_{\infty}}{(zq,\displaystyle\frac{xq}{z};q)_{\infty}}.$ Under the same substitution, the left side becomes
\begin{equation}
\displaystyle\sum_{n\geq 0}(-1)^n\displaystyle\frac{(x;q)_n(1-xq^{2n})}{(1-x)(q;q)_n}\displaystyle\frac{(1-\displaystyle\frac{x}{z})(1-z)}{(1-zq^n)(1-\displaystyle\frac{x}{z}q^n)}q^{\binom{n+1}{2}}.
\end{equation}
Next we replace $\displaystyle\frac{(1-z)(1-\displaystyle\frac{x}{z})(1-xq^{2n})}{(1-zq^n)(1-\displaystyle\frac{x}{z}q^n)}$ by $-(1-q^n)(1-xq^n)\big(\displaystyle\frac{1}{q^n(1-zq^n)}+\displaystyle\frac{x/z}{(1-\displaystyle\frac{xq^n}{z})}\big)+\big(\displaystyle\frac{1-xq^{2n}}{q^n}\big),$
so that (7) becomes 
\begin{equation}
\begin{split}
    \displaystyle\sum_{n\geq 0}(-1)^{n-1}\displaystyle\frac{(x;q)_n}{(1-x)(q;q)_n}q^{\binom{n+1}{2}}\big(\displaystyle\frac{1}{q^n(1-zq^n)}+\displaystyle\frac{x/z}{(1-\displaystyle\frac{xq^n}{z})}\big)(1-xq^n)(1-q^n)\\
    +\displaystyle\sum_{n\geq 0}(-1)^{n}\displaystyle\frac{(x;q)_n}{(1-x)(q;q)_n}q^{\binom{n+1}{2}}\displaystyle\frac{(1-xq^{2n})}{q^n}.
\end{split}
\end{equation}
We next note that the second sum in (8) vanishes. Indeed, if we set $a=x, bc=qx, d\rightarrow\infty$ in (6), we get that 
\begin{equation*}
    \displaystyle\sum_{n\geq 0}(-1)^{n}\displaystyle\frac{(x;q)_n}{(1-x)(q;q)_n}q^{\binom{n-1}{2}}(1-xq^{2n})=\displaystyle\frac{(xq,1;q)_{\infty}}{(\displaystyle\frac{xq}{b},b;q)_{\infty}}=0.
\end{equation*}
Thus, by (6), (7) and (8), we have that
\begin{equation*}
    \displaystyle\sum_{n\geq 0}(-1)^{n-1}\displaystyle\frac{(x;q)_n}{(1-x)(q;q)_n}q^{\binom{n+1}{2}}(1-xq^n)(1-q^n)\big(\displaystyle\frac{1}{q^n(1-zq^n)}+\displaystyle\frac{x/z}{(1-\displaystyle\frac{xq^n}{z})}\big)=\displaystyle\frac{(xq,q;q)_{\infty}}{(zq,\displaystyle\frac{xq}{z};q)_{\infty}}.
\end{equation*}
Theorem 3.1 now follows on dividing both sides by $(q;q)_{\infty}.$
\end{proof}
\begin{proof}[Proof of Corollary 3.2] 
We write $\displaystyle\frac{1}{q^n(1-zq^n)}=q^{-n}(1+zq^n+z^2q^{2n}+...)=\displaystyle\sum_{j\geq 0}z^jq^{n(j-1)}$ and $\displaystyle\frac{x/z}{(1-\displaystyle\frac{xq^n}{z})}=\displaystyle\sum_{j\geq 0}x^{j+1}z^{-j-1}q^{nj}.$ \\ \\
Taking exponents of the form $kt+b$ for $z$ (since we count the partitions with crank congruent to $b$ modulo $k),$ we thus get that 
\begin{equation*}
    \displaystyle\sum_{n\geq 0}M_{\omega}(b,k,n)q^n=\frac{\partial}{\partial x}\Bigr\rvert_{x = 1}\displaystyle\frac{1}{(q;q)_{\infty}}\displaystyle\sum_{n\geq 1}(-1)^{n-1}q^{\binom{n+1}{2}}\displaystyle\frac{(xq;q)_n}{(q;q)_n}\big[\displaystyle\frac{q^{n(b-1)}}{1-q^{nk}}+\displaystyle\frac{x^{k-b}q^{(k-1-b)n}}{1-x^kq^{nk}}\big]
\end{equation*}
which is Corollary 3.2.
\end{proof}
\vspace{1cm}
\begin{proof}[Proof of Corollary 3.3] Set $k=5, b=1,2,3,4$ in Corollary 3.2 to obtain 
\begin{multline*}
        \displaystyle\sum_{n\geq 0}\big[M_{\omega}(1,5,n)+2M_{\omega}(2,5,n)-2M_{\omega}(3,5,n)-M_{\omega}(4,5,n)\big]q^n \\ \\
        =\displaystyle\frac{\partial}{\partial x}\Bigr\rvert_{x = 1}\displaystyle\frac{1}{(q;q)_{\infty}}\displaystyle\sum_{n\geq 1}(-1)^{n-1}\displaystyle\frac{(xq;q)_n}{(q;q)_{n-1}}q^{\binom{n+1}{2}} 
\end{multline*}
\begin{multline*}
        \times\Bigg[\displaystyle\frac{1}{1-q^{5n}}+\displaystyle\frac{x^4q^{3n}}{1-x^5q^{5n}}+\displaystyle\frac{2q^n}{1-q^{5n}}+\displaystyle\frac{2x^3q^{2n}}{1-x^5q^{5n}}-\displaystyle\frac{2q^{2n}}{1-q^{5n}}-\displaystyle\frac{2x^2q^n}{1-x^5q^{5n}}-\displaystyle\frac{q^{3n}}{1-q^{5n}}-\displaystyle\frac{x}{1-x^5q^{5n}}\Bigg] \\ \\
        =\displaystyle\frac{\partial}{\partial x}\Bigr\rvert_{x = 1}\displaystyle\frac{1}{(q;q)_{\infty}}\displaystyle\sum_{n\geq 1}(-1)^{n-1}\displaystyle\frac{(xq;q)_n}{(q;q)_{n-1}}q^{\binom{n+1}{2}}\times \displaystyle\frac{(1-x)}{(1-q^{5n})(1-x^5q^{5n})}\Bigg[1+2q^n(1+x)-2q^{2n}(1+x+x^2)\\ \\
        -q^{3n}(1+x)(1+x^2)+q^{5n}x(1+x)(1+x^2)+2q^{6n}x^2(1+x+x^2)-2q^{7n}x^3(1+x)-q^{8n}x^4\Bigg]\\ \\
        =\displaystyle\frac{1}{(q;q)_{\infty}}\displaystyle\sum_{n\geq 1}(-1)^{n}q^{\binom{n+1}{2}}\displaystyle\frac{(1-q^n)}{(1-q^{5n})^2}\Bigg[1+4q^n-6q^{2n}-4q^{3n}+4q^{5n}+6q^{6n}-4q^{7n}-q^{8n}\Bigg] \\ \\
        \equiv \displaystyle\frac{1}{(q;q)_{\infty}}\displaystyle\sum_{n\geq 1}(-1)^nq^{\binom{n+1}{2}}\displaystyle\frac{(1-q^n)^3(1+q^n)}{(1-q^{5n})}\pmod 5
\end{multline*}
where we have used $$\displaystyle\frac{\partial}{\partial x}\Bigr\rvert_{x = 1}(1-x)F(x,q,z)=-F(1,q,z)$$ in the penultimate step. This proves Corollary 3.3.
\end{proof}
\begin{proof}[Proof of Corollary 3.4]
Exactly as in the proof of Corollary 3.3, we get using Corollary 3.2,
\begin{multline*}
        \displaystyle\sum_{n\geq 0}\big[M_{\omega}(1,7,n)+2M_{\omega}(2,7,n)+3M_{\omega}(3,7,n)-3M_{\omega}(4,7,n)-2M_{\omega}(5,7,n)-M_{\omega}(6,7,n)\big]q^n \\ \\ 
        =\displaystyle\frac{\partial}{\partial x}\Bigr\rvert_{x = 1}\displaystyle\frac{1}{(q;q)_{\infty}}\displaystyle\sum_{n\geq 1}(-1)^{n-1}\displaystyle\frac{(xq;q)_n}{(q;q)_{n-1}}q^{\binom{n+1}{2}} \times 
        \Bigg[\displaystyle\frac{1}{1-q^{7n}}+\displaystyle\frac{x^6q^{5n}}{1-x^7q^{7n}}+\displaystyle\frac{2q^n}{1-q^{7n}}+\displaystyle\frac{2x^5q^{4n}}{1-x^7q^{7n}}
\end{multline*}
\begin{multline*}
        +\displaystyle\frac{3q^{2n}}{1-q^{7n}}+\displaystyle\frac{3x^4q^{3n}}{1-x^7q^{7n}}-\displaystyle\frac{3q^{3n}}{1-q^{7n}}-\displaystyle\frac{3x^3q^{2n}}{1-x^7q^{7n}}-\displaystyle\frac{2q^{4n}}{1-q^{7n}}-\displaystyle\frac{2x^2q^n}{1-x^7q^{7n}}-\displaystyle\frac{q^{5n}}{1-q^{7n}}-\displaystyle\frac{x}{1-x^7q^{7n}}\Bigg] \\ \\
        \equiv \displaystyle\frac{1}{(q;q)_{\infty}}\displaystyle\sum_{n\geq 1}(-1)^nq^{\binom{n+1}{2}}\displaystyle\frac{(1-q^n)^5(1+q^n)}{(1-q^{7n})}\pmod 7
\end{multline*}
where we have simplified as in the proof of Corollary 3.3 and again used $$\displaystyle\frac{\partial}{\partial x}\Bigr\rvert_{x = 1}(1-x)F(x,q,z)=-F(1,q,z).$$ This proves Corollary 3.4.
\end{proof}
\section{The Decomposition modulo 5}
In this section, we attempt to prove certain equalities which should imply Theorem 1.3. We start with the following conjecture.
\begin{conjecture}
\begin{align*}
    \displaystyle\frac{1}{(q;q)_{\infty}}\displaystyle\sum_{n\geq 1}\displaystyle\frac{(-1)^nq^{\binom{n+1}{2}}(1-q^n)^3(1+q^n)}{1-q^{5n}} 
    =-q(q^{25};q^{25})_{\infty}G(q^5)+q^2(q^{25};q^{25})_{\infty}H(q^5)\\ \\+q^3(q^{25};q^{25})_{\infty}\displaystyle\frac{H^2(q^5)}{G(q^5)}
\end{align*}
where $G(q)=\displaystyle\frac{1}{(q;q^5)_{\infty}(q^4;q^5)_{\infty}}$ and $H(q)=\displaystyle\frac{1}{(q^2;q^5)_{\infty}(q^3;q^5)_{\infty}}$ represent the familiar infinite products arising in the context of Rogers Ramanujan identities.
\end{conjecture}

We will next show that Conjecture 4.1 follows from five different identities we get on comparing coefficients of powers of $q$ after properly expanding both sides of 4.1. We also give proofs of four of these five identities. We have not been able to yet prove the fifth identity, but we give different formulations of it and various ways we have attacked it. 
\begin{remark}
We note that, Corollary 3.3 and Conjecture 4.1 would together imply that 
\begin{multline}
        \displaystyle\sum_{n\geq 0}\bigg[M_{\omega}(1,5,n)+2M_{\omega}(2,5,n)-2M_{\omega}(3,5,n)-M_{\omega}(4,5,n)\bigg]q^n\equiv \\ \\
        \displaystyle\frac{1}{(q;q)_{\infty}}\displaystyle\sum_{n\geq 1}(-1)^nq^{\binom{n+1}{2}}\displaystyle\frac{(1-q^n)^3(1+q^n)}{(1-q^{5n})}\pmod 5 \\ \\
        \equiv -q(q^{25};q^{25})_{\infty}G(q^5)+q^2(q^{25};q^{25})_{\infty}H(q^5)+q^3(q^{25};q^{25})_{\infty}\displaystyle\frac{H^2(q^5)}{G(q^5)} \pmod 5
\end{multline}
so that $\displaystyle\sum_{n\geq 0}\bigg[M_{\omega}(1,5,5n+4)+2M_{\omega}(2,5,5n+4)-2M_{\omega}(3,5,5n+4)-M_{\omega}(4,5,5n+4)\bigg]\equiv 0\pmod 5$ since there's no exponent of $q$ which is 4 modulo 5 on the right side of Conjecture 4.1. This means that, in conjunction with Corollary 3.3, a proof of Conjecture 4.1 will yield Theorem 1.3
\end{remark}
\begin{theorem}
Conjecture 4.1 follows from the following five equalities. \begin{equation}
    (i)\hspace{0.5cm} \displaystyle\sum_{m}(-1)^m\displaystyle\frac{q^{\frac{5m^2+3m}{2}}}{1-q^{5m+1}}-\displaystyle\sum_{m}(-1)^m\displaystyle\frac{q^{\frac{5m^2-3m}{2}}}{1-q^{5m-2}}=(q^5;q^5)_{\infty}^2\displaystyle\frac{G^2(q)}{H(q)} 
\end{equation}
\begin{equation}
    (ii) \hspace{0.5cm}\displaystyle\sum_{m}(-1)^m\displaystyle\frac{q^{\frac{5m^2+5m}{2}}}{1-q^{5m+1}}=(q^5;q^5)_{\infty}^2G(q) 
\end{equation}
\begin{equation}
    (iii)\hspace{0.5cm}\displaystyle\sum_{m}(-1)^m\displaystyle\frac{q^{\frac{5m^2+5m}{2}}}{1-q^{5m+2}}=(q^5;q^5)_{\infty}^2H(q) 
\end{equation}
\begin{equation}
    (iv)\hspace{0.5cm}\displaystyle\sum_{m}(-1)^m\displaystyle\frac{q^{\frac{5m^2+m-2}{2}}}{1-q^{5m-1}}-\displaystyle\sum_{m}(-1)^m\displaystyle\frac{q^{\frac{5m^2-m-2}{2}}}{1-q^{5m-2}}=(q^5;q^5)_{\infty}^2\displaystyle\frac{H^2(q)}{G(q)} 
\end{equation}
\begin{equation*}
    (v)\hspace{0.5cm}-\displaystyle\sum_{m}' (-1)^m\displaystyle\frac{q^{\frac{5m^2+m-2}{2}}}{1-q^{5m}}+\displaystyle\sum_{m}(-1)^m\displaystyle\frac{q^{\frac{5m^2-m-2}{2}}}{1-q^{5m-1}}+2\displaystyle\sum_{m}'(-1)^m\displaystyle\frac{q^{\frac{5m^2+3m-2}{2}}}{1-q^{5m}}+
\end{equation*}
\begin{equation}
2\displaystyle\sum_{m}(-1)^m\displaystyle\frac{q^{\frac{5m^2+7m}{2}}}{1-q^{5m+2}}=\\ \\
    (q^5;q^5)_{\infty}^2\displaystyle\frac{H^3(q)}{G^2(q)}
\end{equation}
\end{theorem}
In this theorem and subsequently, we adopt the notation that $\displaystyle\sum_{n}$ means $\displaystyle\sum_{n=-\infty}^{\infty}$ while $\displaystyle\sum_{n}'$ stands for $\displaystyle\sum_{\substack{n=-\infty\\n\neq 0}}^{\infty}.$
\begin{proof}
Recall that Conjecture 4.1 is the following:
\begin{align*}
    \displaystyle\frac{1}{(q;q)_{\infty}}\displaystyle\sum_{n\geq 1}\displaystyle\frac{(-1)^nq^{\binom{n+1}{2}}(1-q^n)^3(1+q^n)}{1-q^{5n}}
    =-q(q^{25};q^{25})_{\infty}G(q^5)+q^2(q^{25};q^{25})_{\infty}H(q^5)
\end{align*}
\begin{align}
+q^3(q^{25};q^{25})_{\infty}\displaystyle\frac{H^2(q^5)}{G(q^5)}
\end{align}
By Lemma 6 in [1], 
\begin{equation*}
    (q;q)_{\infty}=-qP(0)\Bigg[1-q^{-1}\displaystyle\frac{P(2)}{P(1)}+q\displaystyle\frac{P(4)}{P(2)}\Bigg]
\end{equation*}
with $P(0)=(q^{25};q^{25})_{\infty},P(1)=P(4)=(q^5,q^{20};q^{25})_{\infty},P(2)=(q^{10},q^{15};q^{25})_{\infty}.$ \\ \\
Thus, $(q;q)_{\infty}=(q^{25};q^{25})_{\infty}\Bigg(\displaystyle\frac{G(q^5)}{H(q^5)}-q-q^2\displaystyle\frac{H(q^5)}{G(q^5)}\Bigg).$ We get that 
\begin{multline}
(q;q)_{\infty}\Bigg[-q(q^{25};q^{25})_{\infty}G(q^5)+q^2(q^{25};q^{25})_{\infty}H(q^5)+q^3(q^{25};q^{25})_{\infty}\displaystyle\frac{H^2(q^5)}{G(q^5)}\Bigg]=\\ \\(q^{25};q^{25})_{\infty}^2\Bigg(-q\displaystyle\frac{G^2(q^5)}{H(q^5)}+2q^2G(q^5)+q^3H(q^5)-2q^4\displaystyle\frac{H^2(q^5)}{G(q^5)}-q^5\displaystyle\frac{H^3(q^5)}{G^2(q^5)}\Bigg)
\end{multline}
On the other hand, 
\begin{equation*}
    \displaystyle\sum_{n\geq 1}\displaystyle\frac{(-1)^nq^{\binom{n+1}{2}}(1-q^n)^3(1+q^n)}{1-q^{5n}}=\displaystyle\sum_{n\geq 1}\displaystyle\frac{(-1)^nq^{n(n+1)/2}(1-2q^n+2q^{3n}-q^{4n})}{1-q^{5n}}.
\end{equation*}
Since on replacing $n$ with $-n$,
\begin{equation*}
    \displaystyle\sum_{n\geq 1}\displaystyle\frac{(-1)^nq^{n(n+1)/2}q^{3n}}{1-q^{5n}}=-\displaystyle\sum_{n\leq -1}\displaystyle\frac{(-1)^nq^{n(n+1)/2}q^{n}}{1-q^{5n}}
\end{equation*}
and
\begin{equation*}
    \displaystyle\sum_{n\geq 1}\displaystyle\frac{(-1)^nq^{n(n+1)/2}q^{4n}}{1-q^{5n}}=-\displaystyle\sum_{n\leq -1}\displaystyle\frac{(-1)^nq^{n(n+1)/2}}{1-q^{5n}},
\end{equation*}
we get that,
\begin{equation}
\displaystyle\sum_{n\geq 1}\displaystyle\frac{(-1)^nq^{\binom{n+1}{2}}(1-q^n)^3(1+q^n)}{1-q^{5n}}=\displaystyle\sum'_{n}\displaystyle\frac{(-1)^nq^{n(n+1)/2}}{1-q^{5n}}-2\displaystyle\sum'_{n}\displaystyle\frac{(-1)^nq^{n(n+3)/2}}{1-q^{5n}}
\end{equation}
Set $n=5m+t.$ \\ \\
Then 
\begin{align*}
    \displaystyle\sum'_{n}\displaystyle\frac{(-1)^nq^{n(n+1)/2}}{1-q^{5n}}=\displaystyle\sum_{t=-2}^{2}(-1)^tq^{t(t+1)/2}\displaystyle\sum_{m}\displaystyle\frac{(-1)^mq^{((5m+t)(5m+t+1)-t(t+1))/2}}{1-q^{25m+5t}} \\
    =\displaystyle\sum_{t=-2}^{2}(-1)^tq^{t(t+1)/2}\displaystyle\sum_{m}\displaystyle\frac{(-1)^mq^{(25m^2+10mt+5m)/2}}{1-q^{25m+5t}} 
\end{align*}
\hspace{6.5cm}(with the term for $m=t=0$ omitted)
\begin{multline}
    =\displaystyle\sum_{m}'\displaystyle\frac{(-1)^mq^{(25m^2+5m)/2}}{1-q^{25m}}-\displaystyle\sum_{m}\displaystyle\frac{(-1)^mq^{(25m^2+15m+2)/2}}{1-q^{25m+5}}-\displaystyle\sum_{m}\displaystyle\frac{(-1)^mq^{(25m^2-5m)/2}}{1-q^{25m-5}}\\-\displaystyle\sum_{m}\displaystyle\frac{(-1)^mq^{(25m^2+25m+6)/2}}{1-q^{25m+10}}+\displaystyle\sum_{m}\displaystyle\frac{(-1)^mq^{(25m^2-15m+2)/2}}{1-q^{25m-10}}.
\end{multline}
Similarly, 
\begin{multline}
    \displaystyle\sum'_{n}\displaystyle\frac{(-1)^nq^{n(n+3)/2}}{1-q^{5n}}=\displaystyle\sum_{m}'\displaystyle\frac{(-1)^mq^{(25m^2+15m)/2}}{1-q^{25m}}-\displaystyle\sum_{m}\displaystyle\frac{(-1)^mq^{(25m^2+25m+4)/2}}{1-q^{25m+5}}\\ -\displaystyle\sum_{m}\displaystyle\frac{(-1)^mq^{(25m^2+5m-2)/2}}{1-q^{25m-5}}+\displaystyle\sum_{m}\displaystyle\frac{(-1)^mq^{(25m^2+35m+10)/2}}{1-q^{25m+10}}+\displaystyle\sum_{m}\displaystyle\frac{(-1)^mq^{(25m^2-5m-2)/2}}{1-q^{25m-10}}.
\end{multline}
Using (17),(18) and (19), we get that 
\begin{multline}
    \displaystyle\sum_{n\geq 1}\displaystyle\frac{(-1)^nq^{\binom{n+1}{2}}(1-q^n)^3(1+q^n)}{1-q^{5n}}=\\ \\ q\Bigg[-\displaystyle\sum_{m}\displaystyle\frac{(-1)^mq^{(25m^2+15m)/2}}{1-q^{25m+5}}+\displaystyle\sum_{m}\displaystyle\frac{(-1)^mq^{(25m^2-15m)/2}}{1-q^{25m-10}}\Bigg]+2q^2\Bigg[\displaystyle\sum_{m}\displaystyle\frac{(-1)^mq^{(25m^2+25m)/2}}{1-q^{25m+5}}\Bigg]\\ \\ +q^3\Bigg[\displaystyle\sum_{m}\displaystyle\frac{(-1)^mq^{(25m^2+25m)/2}}{1-q^{25m+10}}\Bigg]+2q^4\Bigg[-\displaystyle\sum_{m}\displaystyle\frac{(-1)^mq^{(25m^2+5m-10)/2}}{1-q^{25m-5}}+\displaystyle\sum_{m}\displaystyle\frac{(-1)^mq^{(25m^2-5m-10)/2}}{1-q^{25m-10}}\Bigg]\\ \\ +q^5\Bigg[\displaystyle\sum_{m}'\displaystyle\frac{(-1)^mq^{(25m^2+5m-10)/2}}{1-q^{25m}}-\displaystyle\sum_{m}\displaystyle\frac{(-1)^mq^{(25m^2-5m-10)/2}}{1-q^{25m-5}}-2\displaystyle\sum_{m}'\displaystyle\frac{(-1)^mq^{(25m^2+15m-10)/2}}{1-q^{25m}}\\ \\ -2\displaystyle\sum_{m}\displaystyle\frac{(-1)^mq^{(25m^2+35m)/2}}{1-q^{25m+10}}\Bigg]
\end{multline}
Comparing coefficients of $q,q^2,q^3,q^4,q^5$ between (16) and (20) and replacing $q$ by $q^{1/5}$, the proof is complete.
\end{proof}
\begin{theorem}
Equations (11) and (12) in Theorem 4.2 are valid.
\end{theorem}
\begin{proof}
By Theorem 3.2 in [4], we have
\begin{equation}
    ^6\phi_5\left[\begin{array}{c}
             z, q\sqrt z,-q\sqrt z,a_1,a_2,a_3 \\
             \sqrt z,-\sqrt z,\displaystyle\frac{zq}{a_1},\displaystyle\frac{zq}{a_2},\displaystyle\frac{zq}{a_3}
    \end{array}
;q,\displaystyle\frac{zq}{a_1a_2a_3}\right]=\displaystyle\prod\Bigg[\begin{array}{c}
             zq, \displaystyle\frac{zq}{a_1a_2},\displaystyle\frac{zq}{a_1a_3},\displaystyle\frac{zq}{a_2a_3} \\
             \displaystyle\frac{zq}{a_1},\displaystyle\frac{zq}{a_2},\displaystyle\frac{zq}{a_3},\displaystyle\frac{zq}{a_1a_2a_3}
    \end{array}
    \Bigg]
\end{equation}
The left side equals 
\begin{equation*}
\displaystyle\sum_{m\geq 0} \displaystyle\frac{(z;q)_{m}(1-zq^{2m})}{(1-z)(q;q)_m}\displaystyle\frac{(a_1,a_2,a_3;q)_m}{(\displaystyle\frac{zq}{a_1},\displaystyle\frac{zq}{a_2},\displaystyle\frac{zq}{a_3};q)_m}\Bigg(\displaystyle\frac{zq}{a_1a_2a_3}\Bigg)^m 
\end{equation*}
\begin{equation*}
=\displaystyle\sum_{m\geq 0}(-1)^m(1+q^{5m})q^{5m(m+1)/2}\displaystyle\frac{(q,\displaystyle\frac{1}{q};q^5)_m}{(q^6,q^4;q^5)_m} \text{  (put  } z=1, a_1\rightarrow\infty, a_2=q,a_3=\displaystyle\frac{1}{q},q=q^5)
\end{equation*}
\begin{equation*}
=1+\displaystyle\sum_{m\geq 1}(-1)^m(1+q^{5m})q^{5m(m+1)/2}\displaystyle\frac{(1-q)(1-\displaystyle\frac{1}{q})}{(1-q^{5m+1})(1-q^{5m-1})}
\end{equation*}
\begin{equation*}
=1+\displaystyle\sum_{m\geq 1}(-1)^mq^{5m(m+1)/2}(1-q)(1-q^{-1})\Bigg(\displaystyle\frac{1}{(1-q)(1-q^{5m-1})}+\displaystyle\frac{1}{(1-q^{-1})(1-q^{5m+1})}\Bigg) 
\end{equation*}
\begin{equation*}
=1+\displaystyle\sum_{m\leq -1}\displaystyle\frac{(-1)^mq^{(5m^2+5m)/2}(1-q)}{1-q^{5m+1}}+\displaystyle\sum_{m\geq 1}\displaystyle\frac{(-1)^mq^{(5m^2+5m)/2}(1-q)}{1-q^{5m+1}}
\end{equation*}
\begin{equation*}
=(1-q)\displaystyle\sum_{m}\displaystyle\frac{(-1)^mq^{(5m^2+5m)/2}}{1-q^{5m+1}}.
\end{equation*}
On the other hand, setting $z=1, a_1\rightarrow\infty, a_2=q,a_3=\displaystyle\frac{1}{q},q=q^5$ on the right side of (21), we get $\displaystyle\frac{(q^5;q^5)_{\infty}^2}{(q^4,q^6;q^5)_{\infty}}=(1-q)(q^5;q^5)_{\infty}^2G(q)$ and thus equation (11) is proved. \\ \\
Again, if we set $z=1, a_1\rightarrow\infty, a_2=q^2,a_3=\displaystyle\frac{1}{q^2},q=q^5$ on the left side of (21), we similarly get \\ \\
\begin{equation*}
\displaystyle\sum_{m\geq 0} \displaystyle\frac{(q;q)_{m}(1-zq^{2m})}{(1-z)(q;q)_m}\displaystyle\frac{(a_1,a_2,a_3;q)_m}{(\displaystyle\frac{zq}{a_1},\displaystyle\frac{zq}{a_2},\displaystyle\frac{zq}{a_3};q)_m}\Bigg(\displaystyle\frac{zq}{a_1a_2a_3}\Bigg)^m 
\end{equation*}
\begin{equation*}
=\displaystyle\sum_{m\geq 0}(-1)^m(1+q^{5m})q^{5m(m+1)/2}\displaystyle\frac{(q^2,\displaystyle\frac{1}{q^2};q^5)_m}{(q^7,q^3;q^5)_m}
\end{equation*}
\begin{equation*}
=1+\displaystyle\sum_{m\geq 1}(-1)^m(1+q^{5m})q^{5m(m+1)/2}\displaystyle\frac{(1-q^2)(1-\displaystyle\frac{1}{q^2})}{(1-q^{5m+2})(1-q^{5m-2})}
\end{equation*}
\begin{equation*}
=(1-q^2)\displaystyle\sum_{m}\displaystyle\frac{(-1)^mq^{(5m^2+5m)/2}}{1-q^{5m+2}}
\end{equation*}
exactly as before, whereas under the same substitutions, the right side of (21) equals $\displaystyle\frac{(q^5;q^5)_{\infty}^2}{(q^3,q^7;q^5)_{\infty}}=(1-q^2)(q^5;q^5)_{\infty}^2H(q)$, proving (12).
\end{proof}
\begin{theorem}
Equations (10) and (13) in Theorem 4.2 are valid. 
\end{theorem}
\begin{proof}
We start with the Apple-Lerch series $m(x,q,z)=-\displaystyle\frac{z}{j(z;q)}\displaystyle\sum_{r}\displaystyle\frac{(-1)^rq^{r(r+1)/2}z^r}{1-xzq^r}$ with $j(z;q)=(z,q/z,q;q)_{\infty}.$ \\ \\
Then $m(q^2;q^5;1/q)=-\displaystyle\frac{1}{qj(1/q;q^5)}\displaystyle\sum_{m}\displaystyle\frac{(-1)^mq^{5m(m+1)/2}q^{-m}}{1-q^{5m+1}}=-\displaystyle\frac{1}{qj(1/q;q^5)}S_1$ \\ \\
where $S_1=\displaystyle\sum_{m}\displaystyle\frac{(-1)^mq^{(5m^2+3m)/2}}{1-q^{5m+1}}.$ \\ \\
Similarly, $m(q^2;q^5;1/q^4)=\displaystyle\frac{1}{q^4j(1/q^4;q^5)}S_2$ where $S_2=-\displaystyle\sum_{m}\displaystyle\frac{(-1)^mq^{(5m^2-3m)/2}}{1-q^{5m-2}}.$\\ \\
The left side of (10) is therefore 
\begin{multline}
S_1+S_2=-qj(1/q;q^5)m(q^2,q^5,1/q)+q^4j(1/q^4;q^5)m(q^2,q^5,1/q^4)\\ \\
=(q,q^4,q^5;q^5)_{\infty}\Bigg[m(q^2,q^5,1/q)-m(q^2,q^5,1/q^4)\Bigg]
\end{multline}
as $j(1/q;q^5)=(1/q,q^6,q^5;q^5)_{\infty}=-\displaystyle\frac{1}{q}(q,q^4,q^5;q^5)_{\infty}$ and $j(1/q^4;q^5)=(1/q^4,q^9,q^5;q^5)_{\infty}\\ \\=-\displaystyle\frac{1}{q^4}(q,q^4,q^5;q^5)_{\infty}.$ \\ \\
By lemma 11.3.4 in [5], 
\begin{equation*}
    m(x,q,z_1)-m(x,q,z_0)=\displaystyle\frac{z_0(q;q)_{\infty}^3j(z_1/z_0;q)j(xz_0z_1;q)}{j(z_0;q)j(z_1;q)j(xz_0;q)j(xz_1;q)}
\end{equation*}
whence, setting $x=q^2,z_1=\displaystyle\frac{1}{q},z_0=\displaystyle\frac{1}{q^4},q=q^5,$
\begin{equation*}
    m(q^2,q^5,1/q)-m(q^2,q^5,1/q^4)=\displaystyle\frac{\displaystyle\frac{1}{q^4}(q^5;q^5)_{\infty}^3j(q^3;q^5)j(1/q^3;q^5)}{j(1/q^4;q^5)j(1/q;q^5)j(1/q^2;q^5)j(q;q^5)}
\end{equation*}
\begin{equation*}
    =\displaystyle\frac{(q^2,q^3,q^5;q^5)_{\infty}}{(q,q,q,q^4,q^4,q^4;q^5)_{\infty}}
\end{equation*}
By (22), the proof of equation (10) is now finished. \\ \\
For equation (13), we note similarly that $m(q;q^5;1/q^3)=-\displaystyle\frac{1}{q^3j(1/q^3;q^5)}\displaystyle\sum_{m}\displaystyle\frac{(-1)^mq^{(5m^2-m)/2}}{1-q^{5m-2}}$ and $m(q;q^5;1/q^2)=-\displaystyle\frac{1}{q^2j(1/q^2;q^5)}\displaystyle\sum_{m}\displaystyle\frac{(-1)^mq^{(5m^2+m)/2}}{1-q^{5m-1}}$ whence the left side of (13) reduces to
\begin{multline}
-q^2j(1/q^3;q^5)m(q,q^5,1/q^3)+qj(1/q^2;q^5)m(q,q^5,1/q^2)\\ \\
=\displaystyle\frac{1}{q}(q^2,q^3;q^5)_{\infty}\Bigg[m(q,q^5,1/q^3)-m(q,q^5,1/q^2)\Bigg]
\end{multline}
as $j(1/q^2;q^5)=-\displaystyle\frac{1}{q^2}(q^2,q^3,q^5;q^5)_{\infty},j(1/q^3;q^5)=-\displaystyle\frac{1}{q^3}(q^2,q^3,q^5;q^5)_{\infty}.$ \\ \\
By 11.3.4 in [5], we check as before that (setting $x=q,z_1=\displaystyle\frac{1}{q^3},z_0=\displaystyle\frac{1}{q^4},q=q^5)$,
\begin{equation*}
    m(q,q^5,1/q^3)-m(q,q^5,1/q^2)=\displaystyle\frac{q(q,q^4,q^5;q^5)_{\infty}}{(q^2,q^2,q^2,q^3,q^3,q^3;q^5)_{\infty}}
\end{equation*}
whence by (23), equation (13) is proved. 
\end{proof}

The only equation left to be proven to establish the validity of Conjecture 4.1 and hence of Theorem 1.3 is equation (14) in the statement of Theorem 4.2. We haven't been able to prove this, but the following theorem gives three equivalent versions of this equation.
\begin{theorem}
If (A) is the statement \\ \\
\begin{equation*}
    -\displaystyle\sum_{m}' (-1)^m\displaystyle\frac{q^{\frac{5m^2+m-2}{2}}}{1-q^{5m}}+\displaystyle\sum_{m}(-1)^m\displaystyle\frac{q^{\frac{5m^2-m-2}{2}}}{1-q^{5m-1}}+2\displaystyle\sum_{m}'(-1)^m\displaystyle\frac{q^{\frac{5m^2+3m-2}{2}}}{1-q^{5m}}+
\end{equation*}
\begin{equation*}
2\displaystyle\sum_{m}(-1)^m\displaystyle\frac{q^{\frac{5m^2+7m}{2}}}{1-q^{5m+2}}=\\ \\
    (q^5;q^5)_{\infty}^2\displaystyle\frac{H^3(q)}{G^2(q)},
\end{equation*}
(B) is the statement 
\begin{equation*}
    \displaystyle\frac{1}{q^2}\displaystyle\sum_{j=0}^{\infty}\Bigg(\displaystyle\frac{q^{5j+2}}{1-q^{5j+1}}-\displaystyle\frac{q^{5j+5}}{1-q^{5j+4}}-3\displaystyle\frac{q^{5j+3}}{1-q^{5j+2}}+3\displaystyle\frac{q^{5j+4}}{1-q^{5j+3}}\Bigg)=(q^5;q^5)_{\infty}^2\displaystyle\frac{H^3(q)}{G^2(q)}
\end{equation*}
and (C) is the statement 
\begin{equation*}
    \displaystyle\sum_{m}\displaystyle\frac{q^m(1-q^{m+1})^3}{1-q^{5m+5}}=(q^5;q^5)_{\infty}^2\displaystyle\frac{H^3(q)}{G^2(q)}
\end{equation*}
then (A), (B) and (C) are equivalent.
\end{theorem}
\begin{proof}
We will prove Theorem 4.5 by reducing the left side of (A) to the left side of (B) and the left side of (B) to the left side of (C). \\ \\
Let $T_1=\displaystyle\sum_{m}' (-1)^m\displaystyle\frac{q^{\frac{5m^2+m-2}{2}}}{1-q^{5m}},T_2=\displaystyle\sum_{m}(-1)^m\displaystyle\frac{q^{\frac{5m^2-m-2}{2}}}{1-q^{5m-1}},T_3=\displaystyle\sum_{m}'(-1)^m\displaystyle\frac{q^{\frac{5m^2+3m-2}{2}}}{1-q^{5m}}$\\ \\
$T_4=\displaystyle\sum_{m}(-1)^m\displaystyle\frac{q^{\frac{5m^2+7m}{2}}}{1-q^{5m+2}}.$ \\ \\
Then since $m(q^2,q^5,\displaystyle\frac{1}{q^3})=-\displaystyle\frac{1}{q^3j(1/q^3;q^5)}\displaystyle\sum_{m}\displaystyle\frac{(-1)^mq^{(5m^2-m)/2}}{1-q^{5m-1}},$ we get $T_2=-q^2j(1/q^3;q^5)m(q^2,q^5,\displaystyle\frac{1}{q^3}).$ Similarly, $T_4=-\displaystyle\frac{j(q;q^5)}{q}m(q,q^5,q).$
Unfortunately, $T_1$ and $T_3$ are not proper Apple-Lerch sums since they skip the $m=0$ term. \\ \\
Let $m'(x,q,z)=-\displaystyle\frac{z}{j(z;q)}\displaystyle\sum_{r}'\displaystyle\frac{(-1)^rq^{r(r+1)/2}z^r}{1-xzq^r},$ so $m'(q^2,q^5,\displaystyle\frac{1}{q^2})=-\displaystyle\frac{1}{q^2j(1/q^2;q^5)}\displaystyle\sum_{m}'\displaystyle\frac{(-1)^mq^{(5m^2+m)/2}}{1-q^{5m}}$ and hence, $T_1=-qj(1/q^2;q^5)m'(q^2,q^5,\displaystyle\frac{1}{q^2})=-qj(1/q^2;q^5)\Bigg(\lim_{z\to q^{-2}}m(q^2,q^5,z)+\displaystyle\frac{z}{j(z;q^5)}.\displaystyle\frac{1}{1-q^2z}\Bigg).$ \\ \\
In a similar fashion, we may verify that $T_3=-j(1/q;q^5)\Bigg(\lim_{z\to q^{-1}}m(q,q^5,z)+\displaystyle\frac{z}{j(z;q^5)}.\displaystyle\frac{1}{1-qz}\Bigg).$ \\ \\
Thus the left side of (A)$=(-T_1+T_2)+2(T_3+T_4)$
 \begin{align*}
     =\Bigg[qj(1/q^2;q^5)\Bigg(\lim_{z\to q^{-2}}m(q^2,q^5,z)+\displaystyle\frac{z}{j(z;q^5)}.\displaystyle\frac{1}{1-q^2z}\Bigg)-q^2j(1/q^3;q^5)m(q^2,q^5,\displaystyle\frac{1}{q^3})\Bigg] \\ \\
     +2\Bigg[-j(1/q;q^5)\Bigg(\lim_{z\to q^{-1}}m(q,q^5,z)+\displaystyle\frac{z}{j(z;q^5)}.\displaystyle\frac{1}{1-qz}\Bigg)-\displaystyle\frac{j(q;q^5)}{q}m(q,q^5,q)\Bigg] \\ \\
     =X+2Y \text{   ,say.}
 \end{align*}
 Thus,
 \begin{multline*}
     X=q(\displaystyle\frac{1}{q^2},q^7,q^5;q^5)_{\infty}\Bigg(\lim_{z\to q^{-2}}m(q^2,q^5,z)+\displaystyle\frac{z}{j(z;q^5)}.\displaystyle\frac{1}{1-q^2z}\Bigg)-q^2(\displaystyle\frac{1}{q^3},q^8,q^5;q^5)_{\infty}m(q^2,q^5,\displaystyle\frac{1}{q^3}) \\ \\
     =\displaystyle\frac{1}{q}(q^2,q^3,q^5;q^5)_{\infty}\lim_{z\to q^{-2}}\Bigg[m(q^2,q^5,\displaystyle\frac{1}{q^3})-m(q^2,q^5,z)-\displaystyle\frac{z}{j(z;q^5)}\displaystyle\frac{1}{1-q^2z}\Bigg]
 \end{multline*}
 But, by 11.3.4 of [5], we can check that
 \begin{equation*}
 m(q^2,q^5,\displaystyle\frac{1}{q^3})-m(q^2,q^5,z)=\displaystyle\frac{z(q^5;q^5)_{\infty}^3j(\displaystyle\frac{1}{q^3z};q^5)j(\displaystyle\frac{z}{q};q^5)}{j(z;q^5)j(\displaystyle\frac{1}{q^3};q^5)j(q^2z;q^5)j(\displaystyle\frac{1}{q};q^5)} =\displaystyle\frac{z(q^5,\displaystyle\frac{1}{q^3z},q^8z,\displaystyle\frac{z}{q},\displaystyle\frac{q^6}{z};q^5)_{\infty}}{(z\displaystyle\frac{q^5}{z},\displaystyle\frac{1}{q^3},q^8,q^2z,\displaystyle\frac{q^3}{z},\displaystyle\frac{1}{q},q^6;q^5)_{\infty}}
 \end{equation*}
 so that 
 \begin{multline*}
     X=\displaystyle\frac{1}{q}(q^2,q^3,q^5;q^5)_{\infty}\lim_{z\to q^{-2}}\Bigg[\displaystyle\frac{z(q^5,\displaystyle\frac{1}{q^3z},q^8z,\displaystyle\frac{z}{q},\displaystyle\frac{q^6}{z};q^5)_{\infty}}{(z,\displaystyle\frac{q^5}{z},\displaystyle\frac{1}{q^3},q^8,q^2z,\displaystyle\frac{q^3}{z},\displaystyle\frac{1}{q},q^6;q^5)_{\infty}}-\displaystyle\frac{z}{j(z;q^5)(1-q^2z)}\Bigg]\\ \\
     =\displaystyle\frac{1}{q}(q^2,q^3,q^5;q^5)_{\infty}\lim_{z\to q^{-2}}\Bigg[\displaystyle\frac{z}{(z,\displaystyle\frac{q^5}{z},\displaystyle\frac{1}{q^3},q^8,q^7z,\displaystyle\frac{q^3}{z},\displaystyle\frac{1}{q},q^6,q^5;q^5)_{\infty}}\\ 
     \times\displaystyle\frac{(q^5,q^5,\displaystyle\frac{1}{q^3z},q^8z,\displaystyle\frac{z}{q},\displaystyle\frac{q^6}{z};q^5)_{\infty}-(\displaystyle\frac{1}{q^3},q^8,q^7z,\displaystyle\frac{q^3}{z},\displaystyle\frac{1}{q},q^6;q^5)_{\infty}}{(1-q^2z)}\Bigg] \\ \\
     =\displaystyle\frac{1}{q}(q^2,q^3,q^5;q^5)_{\infty}\displaystyle\frac{1/q^2}{(\displaystyle\frac{1}{q^2},q^7,\displaystyle\frac{1}{q^3},q^8,q^5,q^5,\displaystyle\frac{1}{q},q^6,q^5;q^5)_{\infty}} \\ \\
     \times \lim_{z\to q^{-1}}\displaystyle\frac{(q^5,q^5,\displaystyle\frac{1}{q^2z},q^7z,\displaystyle\frac{z}{q^2},\displaystyle\frac{q^7}{z};q^5)_{\infty}-(\displaystyle\frac{1}{q^3},q^8,q^6z,\displaystyle\frac{q^4}{z},\displaystyle\frac{1}{q},q^6;q^5)_{\infty}}{(1-qz)} \\ \\
     =-q^3\displaystyle\frac{1}{(q,q^2,q^3,q^4,q^5,q^5;q^5)_{\infty}}
     \times \lim_{z\to q^{-1}}\displaystyle\frac{(q^5,q^5,\displaystyle\frac{1}{q^2z},q^7z,\displaystyle\frac{z}{q^2},\displaystyle\frac{q^7}{z};q^5)_{\infty}-(\displaystyle\frac{1}{q^3},q^8,q^6z,\displaystyle\frac{q^4}{z},\displaystyle\frac{1}{q},q^6;q^5)_{\infty}}{(1-qz)}
 \end{multline*}
 Suppose, $f(z)=(q^5,q^5,\displaystyle\frac{1}{q^2z},q^7z,\displaystyle\frac{z}{q^2},\displaystyle\frac{q^7}{z};q^5)_{\infty}.$ Then 
 \begin{multline*}
 f'(z)=(q^5;q^5)_{\infty}^2\Bigg[(q^7z,\displaystyle\frac{z}{q^2},\displaystyle\frac{q^7}{z};q^5)_{\infty}\displaystyle\sum_{j=0}^{\infty}\Bigg(1-\displaystyle\frac{q^{5j}}{q^2z}\Bigg)'\displaystyle\prod_{\substack{n=0\\n\neq j}}^{\infty}\Bigg(1-\displaystyle\frac{q^{5n}}{q^2z}\Bigg) \\ 
 +(\displaystyle\frac{1}{q^2z},\displaystyle\frac{z}{q^2},\displaystyle\frac{q^7}{z};q^5)_{\infty}\displaystyle\sum_{j=0}^{\infty}\Bigg(1-(q^7z)q^{5j}\Bigg)'\displaystyle\prod_{\substack{n=0\\n\neq j}}^{\infty}\Bigg(1-(q^{5n}.q^7z)\Bigg) \\ 
 +(\displaystyle\frac{1}{q^2z},q^7z,\displaystyle\frac{q^7}{z};q^5)_{\infty}\displaystyle\sum_{j=0}^{\infty}\Bigg(1-\displaystyle\frac{z}{q^2}.q^{5j}\Bigg)'\displaystyle\prod_{\substack{n=0\\n\neq j}}^{\infty}\Bigg(1-q^{5n}.\displaystyle\frac{z}{q^2}\Bigg)\Bigg] \\ 
 +(\displaystyle\frac{1}{q^2z},q^7z,\displaystyle\frac{z}{q^2};q^5)_{\infty}\displaystyle\sum_{j=0}^{\infty}\Bigg(1-\displaystyle\frac{q^7}{z}.q^{5j}\Bigg)'\displaystyle\prod_{\substack{n=0\\n\neq j}}^{\infty}\Bigg(1-q^{5n}.\displaystyle\frac{q^7}{z}\Bigg)
 \end{multline*}
 Thus $f'(\displaystyle\frac{1}{q})=(q^5;q^5)_{\infty}^2(\displaystyle\frac{1}{q},q^6,\displaystyle\frac{1}{q^3},q^8;q^5)_{\infty}$ $\displaystyle\sum_{j=0}^{\infty}\Bigg(\displaystyle\frac{\displaystyle\frac{q^{5j-2}}{z^2}}{1-\displaystyle\frac{q^{5j-2}}{z}}+\displaystyle\frac{-q^{5j+7}}{1-q^{5j+7}z}+\displaystyle\frac{-q^{5j-2}}{1-q^{5j-2}z}+\displaystyle\frac{\displaystyle\frac{q^{5j+7}}{z^2}}{1-\displaystyle\frac{q^{5j+7}}{z}}\Bigg)_{z=1/q}$ \\ \\
 $=\displaystyle\frac{(q^5;q^5)_{\infty}^2}{q^4}(q,q^2,q^3,q^4;q^5)_{\infty}\displaystyle\sum_{j=0}^{\infty}\Bigg(\displaystyle\frac{q^{5j}}{1-q^{5j-1}}-\displaystyle\frac{q^{5j+7}}{1-q^{5j+6}}-\displaystyle\frac{q^{5j-2}}{1-q^{5j-3}}+\displaystyle\frac{q^{5j+9}}{1-q^{5j+8}}\Bigg).$ \\ \\
 And if $g(z)=(\displaystyle\frac{1}{q^3},q^8,\displaystyle\frac{1}{q},q^6;q^5)_{\infty}(q^6z,\displaystyle\frac{q^4}{z};q^5)_{\infty},$ then we can similarly check that $g'(\displaystyle\frac{1}{q})=0.$ \\ \\
 Thus, by L'Hospitals rule, \\ \\ $X=\displaystyle\frac{1}{q^2}\displaystyle\sum_{j=0}^{\infty}\Bigg(\displaystyle\frac{q^{5j}}{1-q^{5j-1}}-\displaystyle\frac{q^{5j+7}}{1-q^{5j+6}}-\displaystyle\frac{q^{5j-2}}{1-q^{5j-3}}+\displaystyle\frac{q^{5j+9}}{1-q^{5j+8}}\Bigg).$ \\ \\
 Similarly, 
 \begin{multline*}
     Y=-(\displaystyle\frac{1}{q},q^6,q^5;q^5)_{\infty}\Bigg(\lim_{z\to q^{-1}}m(q,q^5,z)+\displaystyle\frac{z}{j(z;q^5)}.\displaystyle\frac{1}{1-qz}\Bigg)-\displaystyle\frac{(q,q^4,q^5;q^5)_{\infty}}{q}m(q,q^5,q) \\ \\
     =-\displaystyle\frac{1}{q}(q,q^4,q^5;q^5)_{\infty}\lim_{z\to q^{-1}}\Bigg[m(q,q^5,q)-m(q^2,q^5,z)-\displaystyle\frac{z}{j(z;q^5)}\displaystyle\frac{1}{1-qz}\Bigg]
 \end{multline*}
 Exactly as before, we can check that $m(q,q^5;q)-m(q,q^5,z)=\displaystyle\frac{z(\displaystyle\frac{q}{z},q^4z,q^2z,\displaystyle\frac{q^3}{z},q^5;q^5)_{\infty}}{(z,\displaystyle\frac{q^5}{z},q,q^4,\displaystyle\frac{q^4}{z},qz,q^2,q^3;q^5)_{\infty}}$ so that 
 \begin{multline*}
 Y=-\displaystyle\frac{1}{q}(q,q^4,q^5;q^5)_{\infty}\lim_{z\to q^{-1}}\Bigg[\displaystyle\frac{z(\displaystyle\frac{q}{z},q^4z,q^2z,\displaystyle\frac{q^3}{z},q^5;q^5)_{\infty}}{(z,\displaystyle\frac{q^5}{z},q,q^4,\displaystyle\frac{q^4}{z},qz,q^2,q^3;q^5)_{\infty}}-\displaystyle\frac{z}{j(z;q^5)(1-qz)}\Bigg]\\ \\
 =-\displaystyle\frac{1}{q}(q,q^4,q^5;q^5)_{\infty}\lim_{z\to q^{-1}}\Bigg[\displaystyle\frac{z}{(z,\displaystyle\frac{q^5}{z},q,q^4,\displaystyle\frac{q^4}{z},q^6z,q^2,q^3,q^5;q^5)_{\infty}} \\ \\
\times\displaystyle\frac{(\displaystyle\frac{q}{z},q^4z,q^2z,\displaystyle\frac{q^3}{z},q^5,q^5;q^5)_{\infty}-(q,q^4,\displaystyle\frac{q^4}{z},q^6z,q^2,q^3;q^5)_{\infty}}{(1-qz)}\Bigg] \\ \\
=\displaystyle\frac{1}{q(q,q^2,q^3,q^4,q^5,q^5;q^5)_{\infty}}\times \lim_{z\to q^{-1}}\displaystyle\frac{(\displaystyle\frac{q}{z},q^4z,q^2z,\displaystyle\frac{q^3}{z},q^5,q^5;q^5)_{\infty}-(q,q^4,\displaystyle\frac{q^4}{z},q^6z,q^2,q^3;q^5)_{\infty}}{(1-qz)}
\end{multline*}
 after simplification as before. \\ \\
As before, if $h(z)=(\displaystyle\frac{q}{z},q^4z,q^2z,\displaystyle\frac{q^3}{z},q^5,q^5;q^5)_{\infty},$ then \\ \\
$h'(\displaystyle\frac{1}{q})=(q^5;q^5)_{\infty}^2(q,q^2,q^3,q^4;q^5)_{\infty}\displaystyle\sum_{j=0}^{\infty}\Bigg(\displaystyle\frac{q^{5j+3}}{1-q^{5j+2}}-\displaystyle\frac{q^{5j+4}}{1-q^{5j+3}}-\displaystyle\frac{q^{5j+2}}{1-q^{5j+1}}+\displaystyle\frac{q^{5j+5}}{1-q^{5j+4}}\Bigg)$ \\ \\
and if $p(z)=(q,q^4,\displaystyle\frac{q^4}{z},q^6z,q^2,q^3;q^5)_{\infty},$ then $p'(\displaystyle\frac{1}{q})=0$ \\ \\
so that $Y=-\displaystyle\frac{1}{q^2}\displaystyle\sum_{j=0}^{\infty}\Bigg(\displaystyle\frac{q^{5j+3}}{1-q^{5j+2}}-\displaystyle\frac{q^{5j+4}}{1-q^{5j+3}}-\displaystyle\frac{q^{5j+2}}{1-q^{5j+1}}+\displaystyle\frac{q^{5j+5}}{1-q^{5j+4}}\Bigg).$ \\ \\
Thus, the left side of (A)=$X+2Y$
\begin{multline*}
    =\displaystyle\frac{1}{q^2}\displaystyle\sum_{j=0}^{\infty}\Bigg(\displaystyle\frac{q^{5j}}{1-q^{5j-1}}-\displaystyle\frac{q^{5j+7}}{1-q^{5j+6}}-\displaystyle\frac{q^{5j-2}}{1-q^{5j-3}}+\displaystyle\frac{q^{5j+9}}{1-q^{5j+8}}\Bigg)\\ \\
    -\displaystyle\frac{2}{q^2}\displaystyle\sum_{j=0}^{\infty}\Bigg(\displaystyle\frac{q^{5j+3}}{1-q^{5j+2}}-\displaystyle\frac{q^{5j+4}}{1-q^{5j+3}}-\displaystyle\frac{q^{5j+2}}{1-q^{5j+1}}+\displaystyle\frac{q^{5j+5}}{1-q^{5j+4}}\Bigg) \\ \\
    =\displaystyle\frac{1}{q^2}\displaystyle\sum_{j=0}^{\infty}\Bigg(\displaystyle\frac{q^{5j+5}}{1-q^{5j+4}}-\displaystyle\frac{q^{5j+2}}{1-q^{5j+1}}-\displaystyle\frac{q^{5j+3}}{1-q^{5j+2}}+\displaystyle\frac{q^{5j+4}}{1-q^{5j+3}}\Bigg)+\displaystyle\frac{1}{q^2}\Bigg(\displaystyle\frac{1}{1-1/q}+\displaystyle\frac{q^2}{1-q}-\displaystyle\frac{1/q^2}{1-1/q^3}-\displaystyle\frac{q^4}{1-q^3}\Bigg) \\ \\
    -\displaystyle\frac{2}{q^2}\displaystyle\sum_{j=0}^{\infty}\Bigg(\displaystyle\frac{q^{5j+3}}{1-q^{5j+2}}-\displaystyle\frac{q^{5j+4}}{1-q^{5j+3}}-\displaystyle\frac{q^{5j+2}}{1-q^{5j+1}}+\displaystyle\frac{q^{5j+5}}{1-q^{5j+4}}\Bigg)\\ \\
    =\displaystyle\frac{1}{q^2}\displaystyle\sum_{j=0}^{\infty}\Bigg(\displaystyle\frac{q^{5j+2}}{1-q^{5j+1}}-\displaystyle\frac{q^{5j+5}}{1-q^{5j+4}}-3\displaystyle\frac{q^{5j+3}}{1-q^{5j+2}}+3\displaystyle\frac{q^{5j+4}}{1-q^{5j+3}}\Bigg)
\end{multline*}
which is the left side of (B). This shows that (B) implies (A).\\ \\
Next we note that $\displaystyle\sum_{j=0}^{\infty}\displaystyle\frac{q^{5j}}{1-q^{5j+A}}=\displaystyle\sum_{j,m\geq 0}q^{5j+5jm+Am}=\displaystyle\sum_{m=0}^{\infty}\displaystyle\frac{q^{Am}}{1-q^{5m+5}}.$ \\ \\
Thus the left side of (B) equals $\displaystyle\sum_{m\geq 0}\displaystyle\frac{q^m-q^3.q^{4m}-3q.q^{2m}+3q^2.q^{3m}}{1-q^{5m+5}}=\displaystyle\sum_{m\geq 0}\displaystyle\frac{q^m(1-q^{m+1})^3}{1-q^{5m+5}}$\\ \\
which shows that (B) implies (C) and the proof of Theorem 4.5 is complete.
\end{proof}
\section{The Decomposition modulo 7}
In this section, we will study Theorem 1.4 by again connecting it to the proof of certain equalities we obtain as in the last section. We start with the following conjecture.
\begin{conjecture}
$\displaystyle\frac{1}{(q;q)_{\infty}}\displaystyle\sum_{n\geq 1}\displaystyle\frac{(-1)^nq^{n(n+1)/2}(1-q^n)^5(1+q^n)}{1-q^{7n}}=-qA+3q^2B-2q^3C+q^4D-3q^6E$
\end{conjecture}
where $A=\displaystyle\frac{(q^{49};q^{49})_{\infty}}{(q^{7},q^{42};q^{49})_{\infty}},B=\displaystyle\frac{(q^{14},q^{35},q^{49};q^{49})_{\infty}}{(q^{7},q^{21},q^{28},q^{42};q^{49})_{\infty}},C=\displaystyle\frac{(q^{49};q^{49})_{\infty}}{(q^{14},q^{35};q^{49})_{\infty}}, D=\displaystyle\frac{(q^{49};q^{49})_{\infty}}{(q^{21},q^{28};q^{49})_{\infty}},\\ \\ E=\displaystyle\frac{(q^{7},q^{42},q^{49};q^{49})_{\infty}}{(q^{14},q^{21},q^{28},q^{35};q^{49})_{\infty}}.$\\ \\
We set $L(q)=\displaystyle\frac{1}{(q,q^6;q^7)_{\infty}},N(q)=\displaystyle\frac{1}{(q^2,q^5;q^7)_{\infty}},Q(q)=\displaystyle\frac{1}{(q^3,q^4;q^7)_{\infty}}$ so that Conjecture 5.1 becomes 
\begin{multline*}
    \displaystyle\frac{1}{(q;q)_{\infty}}\displaystyle\sum_{n\geq 1}\displaystyle\frac{(-1)^nq^{n(n+1)/2}(1-q^n)^5(1+q^n)}{1-q^{7n}}\\ \\
    =(q^{49};q^{49})_{\infty}\Bigg[-qL(q^7)+3q^2\displaystyle\frac{L(q^7)N(q^7)}{Q(q^7)}-2q^3N(q^7)+q^4Q(q^7)-3q^6\displaystyle\frac{N(q^7)Q(q^7)}{L(q^7)}\Bigg]
\end{multline*}
\begin{remark}
By Corollary 3.4, Conjecture 5.1 would imply that 
\begin{multline*}
    \displaystyle\sum_{n\geq 0}\bigg[M_{\omega}(1,7,n)+2M_{\omega}(2,7,n)-+3M_{\omega}(3,7,n)-3M_{\omega}(4,7,n)-2M_{\omega}(5,7,n)-M_{\omega}(6,7,n)\bigg]q^n \\ \\  \equiv\displaystyle\frac{1}{(q;q)_{\infty}}\displaystyle\sum_{n\geq 1}(-1)^nq^{\binom{n+1}{2}}\displaystyle\frac{(1-q^n)^5(1+q^n)}{(1-q^{7n})}\pmod 7 \\ \\
    =(q^{49};q^{49})_{\infty}\Bigg[-qL(q^7)+3q^2\displaystyle\frac{L(q^7)Q(q^7)}{N(q^7)}-2q^3N(q^7)+q^4Q(q^7)-3q^6\displaystyle\frac{N(q^7)Q(q^7)}{L(q^7)}\Bigg]\pmod 7
\end{multline*}
which would mean \\ \\$\displaystyle\sum_{n\geq 0}\bigg[M_{\omega}(1,7,7n+5)+2M_{\omega}(2,7,7n+5)+3M_{\omega}(3,7,7n+5)-3M_{\omega}(4,7,7n+5)\\-2M_{\omega}(5,7,7n+5)-M_{\omega}(6,7,7n+5)\bigg]\equiv 0\pmod 7$ which is Theorem 1.4.

\end{remark}
We will now show that Conjecture 5.1 follows from seven equalities we obtain by comparing coefficients of powers of $q$ by properly expanding both sides of the conjecture. As in the modulo 5 case, we also give proofs of six of these seven equalities, and give various forms we have obtained for the seventh.
\begin{theorem}
Conjecture 5.1 follows from the following seven equalities. 
\begin{equation}
    (i)\hspace{0.5cm} -\displaystyle\sum_{m}(-1)^m\displaystyle\frac{q^{\frac{7m^2+3m}{2}}}{1-q^{7m+1}}+\displaystyle\sum_{m}(-1)^m\displaystyle\frac{q^{\frac{7m^2-3m}{2}}}{1-q^{7m-2}}=(q^7;q^7)_{\infty}^2\Bigg(-\displaystyle\frac{L^2(q)}{N(q)}+q\displaystyle\frac{Q(q)N(q)}{L(q)}\Bigg)
\end{equation}
\begin{multline}
    (ii)\hspace{0.5cm} 4\displaystyle\sum_{m}(-1)^m\displaystyle\frac{q^{\frac{7m^2+5m}{2}}}{1-q^{7m+1}}+4\displaystyle\sum_{m}(-1)^m\displaystyle\frac{q^{\frac{7m^2+9m+2}{2}}}{1-q^{7m+3}}\\ \\
    =(q^7;q^7)_{\infty}^2\Bigg(\displaystyle\frac{L(q)N(q)}{Q(q)}+3\displaystyle\frac{L^2(q)Q(q)}{N^2(q)}+q\displaystyle\frac{Q^2(q)}{L(q)}\Bigg)
\end{multline}
\begin{multline}
    (iii)\hspace{0.5cm} \displaystyle\sum_{m}(-1)^m\displaystyle\frac{q^{\frac{7m^2+5m}{2}}}{1-q^{7m+2}}-\displaystyle\sum_{m}(-1)^m\displaystyle\frac{q^{\frac{7m^2-5m}{2}}}{1-q^{7m-3}}-5\displaystyle\sum_{m}(-1)^m\displaystyle\frac{q^{\frac{7m^2+7m}{2}}}{1-q^{7m+1}}=-4(q^7;q^7)_{\infty}^2L(q)
\end{multline}
\begin{multline}
    (iv)\hspace{0.5cm} 5\displaystyle\sum_{m}(-1)^m\displaystyle\frac{q^{\frac{7m^2+m-2}{2}}}{1-q^{7m-2}}-5\displaystyle\sum_{m}(-1)^m\displaystyle\frac{q^{\frac{7m^2-m-2}{2}}}{1-q^{7m-3}}\\ \\
    =(q^7;q^7)_{\infty}^2\Bigg(2\displaystyle\frac{N^2(q)}{Q(q)}-2\displaystyle\frac{L(q)Q(q)}{N(q)}-3q\displaystyle\frac{Q^2(q)N(q)}{L^2(q)}\Bigg)
\end{multline}
\begin{multline}
    (v)\hspace{0.5cm} -4\displaystyle\sum_{m}(-1)^m\displaystyle\frac{q^{\frac{7m^2+7m}{2}}}{1-q^{7m+2}}-5\displaystyle\sum_{m}(-1)^m\displaystyle\frac{q^{\frac{7m^2+3m-2}{2}}}{1-q^{7m-1}}-5\displaystyle\sum_{m}(-1)^m\displaystyle\frac{q^{\frac{7m^2+11m+2}{2}}}{1-q^{7m+3}}=(q^7;q^7)_{\infty}^2N(q)
\end{multline}
\begin{multline}
    (vi)\hspace{0.5cm} -\displaystyle\sum_{m}(-1)^m\displaystyle\frac{q^{\frac{7m^2+7m}{2}}}{1-q^{7m+3}}+4\displaystyle\sum_{m}(-1)^m\displaystyle\frac{q^{\frac{7m^2+m-2}{2}}}{1-q^{7m-1}}-4\displaystyle\sum_{m}(-1)^m\displaystyle\frac{q^{\frac{7m^2-m-2}{2}}}{1-q^{7m-2}}=-5(q^7;q^7)_{\infty}^2Q(q)
\end{multline}
\begin{multline}
    (vii)\hspace{0.5cm} \displaystyle\sum_{m}'(-1)^m\displaystyle\frac{q^{\frac{7m^2+m-2}{2}}}{1-q^{7m}}-\displaystyle\sum_{m}(-1)^m\displaystyle\frac{q^{\frac{7m^2-m-2}{2}}}{1-q^{7m-1}}-4\displaystyle\sum_{m}'(-1)^m\displaystyle\frac{q^{\frac{7m^2+3m-2}{2}}}{1-q^{7m}}\\ \\
    +4\displaystyle\sum_{m}(-1)^m\displaystyle\frac{q^{\frac{7m^2-3m-2}{2}}}{1-q^{7m-3}}+5\displaystyle\sum_{m}'(-1)^m\displaystyle\frac{q^{\frac{7m^2+5m-2}{2}}}{1-q^{7m}}+5\displaystyle\sum_{m}(-1)^m\displaystyle\frac{q^{\frac{7m^2+9m}{2}}}{1-q^{7m+2}}\\ \\
    =3(q^7;q^7)_{\infty}^2\Bigg(\displaystyle\frac{Q^2(q)}{N(q)}+\displaystyle\frac{N^2(q)}{L(q)}\Bigg)
\end{multline}
\end{theorem}
First we prove the following lemma.
\begin{lemma}
\begin{equation*}
    (q;q)_{\infty}=(q^{49};q^{49})_{\infty}\Bigg[\displaystyle\frac{L(q^7)}{N(q^7)}-q\displaystyle\frac{N(q^7)}{Q(q^7)}-q^2+q^5\displaystyle\frac{Q(q^7)}{L(q^7)}\Bigg]
\end{equation*}
\end{lemma}
\begin{proof}
By Lemma 6 in [1], 
\begin{equation*}
    (q;q)_{\infty}=-q^2P(0)\Bigg[1-q^{-2}\displaystyle\frac{P(2)}{P(1)}+q^{-1}\displaystyle\frac{P(4)}{P(2)}-q^3\displaystyle\frac{P(6)}{P(3)}\Bigg]
\end{equation*}
with $P(0)=(q^{49};q^{49})_{\infty},P(1)=P(6)=(q^7,q^{42};q^{49})_{\infty},P(2)=(q^{14},q^{35};q^{49})_{\infty},P(3)=P(4)=(q^{21},q^{28};q^{49})_{\infty}.$ \\ \\
Thus, 
\begin{multline*}
    (q;q)_{\infty}=P(0)\Bigg[\displaystyle\frac{P(2)}{P(1)}-q^\displaystyle\frac{P(4)}{P(2)}-q^2+q^5\displaystyle\frac{P(6)}{P(3)}\Bigg]=(q^{49};q^{49})_{\infty}\Bigg[\displaystyle\frac{L(q^7)}{N(q^7)}-q\displaystyle\frac{N(q^7)}{Q(q^7)}-q^2+q^5\displaystyle\frac{Q(q^7)}{L(q^7)}\Bigg]
\end{multline*}
as desired.
\end{proof}
\begin{proof}[Proof of Theorem 5.2]
Using the expansion of $(q;q)_{\infty}$ from Lemma 5.3, Conjecture 5.1 becomes 
\begin{multline}
    \displaystyle\sum_{n\geq 1}\displaystyle\frac{(-1)^nq^{n(n+1)/2}(1-q^n)^5(1+q^n)}{1-q^{7n}}\\ \\
    =(q^{49};q^{49})_{\infty}^2\Bigg[-qL(q^7)+3q^2\displaystyle\frac{L(q^7)Q(q^7)}{N(q^7)}-2q^3N(q^7)+q^4Q(q^7)-3q^6\displaystyle\frac{N(q^7)Q(q^7)}{L(q^7)}\Bigg]\\ \\
    \times \Bigg[\displaystyle\frac{L(q^7)}{N(q^7)}-q\displaystyle\frac{N(q^7)}{Q(q^7)}-q^2+q^5\displaystyle\frac{Q(q^7)}{L(q^7)}\Bigg] \\ \\
    =(q^{49};q^{49})_{\infty}^2\Bigg[-q\displaystyle\frac{L^2(q^7)}{N(q^7)}+q^2\Bigg(\displaystyle\frac{L(q^7)N(q^7)}{Q(q^7)}+3\displaystyle\frac{L^2(q^7)Q(q^7)}{N^2(q^7)}\Bigg)-4q^3L(q^7)\\ \\
    +2q^4\Bigg(\displaystyle\frac{N^2(q^7)}{Q(q^7)}-\displaystyle\frac{L(q^7)Q(q^7)}{N(q^7)}\Bigg)+q^5N(q^7)-5q^6Q(q^7)+3q^7\Bigg(\displaystyle\frac{Q^2(q^7)}{N(q^7)}+\displaystyle\frac{N^2(q^7)}{L(q^7)}\Bigg)\\ \\
    +q^8\displaystyle\frac{Q(q^7)N(q^7)}{L(q^7)}+q^9\displaystyle\frac{Q^2(q^7)}{L(q^7)}-3q^{11}\displaystyle\frac{Q^2(q^7)N(q^7)}{L^2(q^7)}\Bigg]
\end{multline}
The left side in (31) is \\ \\$\displaystyle\sum_{n\geq 1}\displaystyle\frac{(-1)^nq^{n(n+1)/2}}{1-q^{7n}}\Bigg(1-4q^n+5q^{2n}-5q^{4n}+4q^{5n}-q^{6n}\Bigg)$\\ \\
\begin{equation} 
=\displaystyle\sum_{n}'\displaystyle\frac{(-1)^nq^{n(n+1)/2}}{1-q^{7n}}-4\displaystyle\sum_{n}'\displaystyle\frac{(-1)^nq^{n(n+3)/2}}{1-q^{7n}}+5\displaystyle\sum_{n}'\displaystyle\frac{(-1)^nq^{n(n+5)/2}}{1-q^{7n}}
\end{equation}
since $-\displaystyle\sum_{n\geq 1}\displaystyle\frac{(-1)^nq^{n(n+1)/2}}{1-q^{7n}}.q^{6n}=\displaystyle\sum_{n\leq -1}\displaystyle\frac{(-1)^nq^{n(n+1)/2}}{1-q^{7n}},$ $\displaystyle\sum_{n\geq 1}\displaystyle\frac{(-1)^nq^{n(n+1)/2}}{1-q^{7n}}.q^{5n}=-\displaystyle\sum_{n\leq -1}\displaystyle\frac{(-1)^nq^{n(n+3)/2}}{1-q^{7n}}$ \\ \\ and $-\displaystyle\sum_{n\geq 1}\displaystyle\frac{(-1)^nq^{n(n+1)/2}}{1-q^{7n}}.q^{4n}=\displaystyle\sum_{n\leq -1}\displaystyle\frac{(-1)^nq^{n(n+5)/2}}{1-q^{7n}}.$ \\ \\
Now, 
\begin{equation*}
    \displaystyle\sum_{n}'\displaystyle\frac{(-1)^nq^{n(n+1)/2}}{1-q^{7n}}=\displaystyle\sum_{t=-3}^{3}(-1)^tq^{t(t+1)/2}\displaystyle\sum_{m}\displaystyle\frac{(-1)^mq^{((7m+t)(7m+t+1)-t(t+1))/2}}{1-q^{49m+7t}},
\end{equation*}
(with the $m=t=0$ term omitted)
\begin{multline} 
    =\displaystyle\sum_{m}'\displaystyle\frac{(-1)^mq^{(49m^2+7m)/2}}{1-q^{49m}}-\displaystyle\sum_{m}\displaystyle\frac{(-1)^mq^{(49m^2+21m+2)/2}}{1-q^{49m+7}}-\displaystyle\sum_{m}\displaystyle\frac{(-1)^mq^{(49m^2-7m)/2}}{1-q^{49m-7}}\\ \\
    +\displaystyle\sum_{m}\displaystyle\frac{(-1)^mq^{(49m^2+35m+6)/2}}{1-q^{49m+14}}+\displaystyle\sum_{m}\displaystyle\frac{(-1)^mq^{(49m^2-21m+2)/2}}{1-q^{49m-14}}-\displaystyle\sum_{m}\displaystyle\frac{(-1)^mq^{(49m^2+49m+12)/2}}{1-q^{49m+21}}\\ \\
    -\displaystyle\sum_{m}\displaystyle\frac{(-1)^mq^{(49m^2-35m+6)/2}}{1-q^{49m-21}}
\end{multline}
Similarly, we get, $\displaystyle\sum_{n}'\displaystyle\frac{(-1)^nq^{n(n+3)/2}}{1-q^{7n}}$
\begin{multline}
    =\displaystyle\sum_{m}'\displaystyle\frac{(-1)^mq^{(49m^2+21m)/2}}{1-q^{49m}}-\displaystyle\sum_{m}\displaystyle\frac{(-1)^mq^{(49m^2+35m+4)/2}}{1-q^{49m+7}}-\displaystyle\sum_{m}\displaystyle\frac{(-1)^mq^{(49m^2+7m-2)/2}}{1-q^{49m-7}}\\ \\
    +\displaystyle\sum_{m}\displaystyle\frac{(-1)^mq^{(49m^2+49m+10)/2}}{1-q^{49m+14}}+\displaystyle\sum_{m}\displaystyle\frac{(-1)^mq^{(49m^2-7m-2)/2}}{1-q^{49m-14}}-\displaystyle\sum_{m}\displaystyle\frac{(-1)^mq^{(49m^2+63m+18)/2}}{1-q^{49m+21}}\\ \\
    -\displaystyle\sum_{m}\displaystyle\frac{(-1)^mq^{(49m^2-21m)/2}}{1-q^{49m-21}}
\end{multline}
and $\displaystyle\sum_{n}'\displaystyle\frac{(-1)^nq^{n(n+5)/2}}{1-q^{7n}}$
\begin{multline}
    =\displaystyle\sum_{m}'\displaystyle\frac{(-1)^mq^{(49m^2+35m)/2}}{1-q^{49m}}-\displaystyle\sum_{m}\displaystyle\frac{(-1)^mq^{(49m^2+49m+6)/2}}{1-q^{49m+7}}-\displaystyle\sum_{m}\displaystyle\frac{(-1)^mq^{(49m^2+21m-4)/2}}{1-q^{49m-7}}\\ \\
    +\displaystyle\sum_{m}\displaystyle\frac{(-1)^mq^{(49m^2+63m+14)/2}}{1-q^{49m+14}}+\displaystyle\sum_{m}\displaystyle\frac{(-1)^mq^{(49m^2+7m-6)/2}}{1-q^{49m-14}}-\displaystyle\sum_{m}\displaystyle\frac{(-1)^mq^{(49m^2+77m+24)/2}}{1-q^{49m+21}}\\ \\
    -\displaystyle\sum_{m}\displaystyle\frac{(-1)^mq^{(49m^2-7m-6)/2}}{1-q^{49m-21}}
\end{multline}
Thus, by (32) through (35) the left side in (31) equals 
\begin{multline*}
    \Bigg[\displaystyle\sum_{m}'(-1)^m\displaystyle\frac{q^{\frac{49m^2+7m}{2}}}{1-q^{49m}}-\displaystyle\sum_{m}(-1)^m\displaystyle\frac{q^{\frac{49m^2-7m}{2}}}{1-q^{49m-7}}-4\displaystyle\sum_{m}'(-1)^m\displaystyle\frac{q^{\frac{49m^2+21m}{2}}}{1-q^{49m}}\\ \\
    +4\displaystyle\sum_{m}(-1)^m\displaystyle\frac{q^{\frac{49m^2-21m}{2}}}{1-q^{49m-21}}+5\displaystyle\sum_{m}'(-1)^m\displaystyle\frac{q^{\frac{49m^2+35m}{2}}}{1-q^{49m}}+5\displaystyle\sum_{m}(-1)^m\displaystyle\frac{q^{\frac{49m^2+63m+14}{2}}}{1-q^{49m+14}}\Bigg] 
\end{multline*} 
\begin{multline}
    +q\Bigg[-\displaystyle\sum_{m}(-1)^m\displaystyle\frac{q^{\frac{49m^2+21m}{2}}}{1-q^{49m+7}}+\displaystyle\sum_{m}(-1)^m\displaystyle\frac{q^{\frac{49m^2-21m}{2}}}{1-q^{49m-14}}\Bigg] \\ \\
    +4q^2\Bigg[\displaystyle\sum_{m}(-1)^m\displaystyle\frac{q^{\frac{49m^2+35m}{2}}}{1-q^{49m+7}}+\displaystyle\sum_{m}(-1)^m\displaystyle\frac{q^{\frac{49m^2+63m+14}{2}}}{1-q^{49m+21}}\Bigg] \\ \\
    +q^3\Bigg[\displaystyle\sum_{m}(-1)^m\displaystyle\frac{q^{\frac{49m^2+35m}{2}}}{1-q^{49m+14}}-\displaystyle\sum_{m}(-1)^m\displaystyle\frac{q^{\frac{49m^2-35m}{2}}}{1-q^{49m-21}}-5\displaystyle\sum_{m}(-1)^m\displaystyle\frac{q^{\frac{49m^2+49m}{2}}}{1-q^{49m+7}}\Bigg] \\ \\
    +5q^4\bigg[\displaystyle\sum_{m}(-1)^m\displaystyle\frac{q^{\frac{49m^2+7m-14}{2}}}{1-q^{49m-14}}-\displaystyle\sum_{m}(-1)^m\displaystyle\frac{q^{\frac{49m^2-7m-14}{2}}}{1-q^{49m-21}}\Bigg] \\ \\
    +q^5\Bigg[-4\displaystyle\sum_{m}(-1)^m\displaystyle\frac{q^{\frac{49m^2+49m}{2}}}{1-q^{49m+14}}-5\displaystyle\sum_{m}(-1)^m\displaystyle\frac{q^{\frac{49m^2+21m-14}{2}}}{1-q^{49m-7}}-5\displaystyle\sum_{m}(-1)^m\displaystyle\frac{q^{\frac{49m^2+77m+14}{2}}}{1-q^{49m+21}}\Bigg] \\ \\
    +q^6\Bigg[-\displaystyle\sum_{m}(-1)^m\displaystyle\frac{q^{\frac{49m^2+49m}{2}}}{1-q^{49m+21}}+4\displaystyle\sum_{m}(-1)^m\displaystyle\frac{q^{\frac{49m^2+7m-14}{2}}}{1-q^{49m-7}}-4\displaystyle\sum_{m}(-1)^m\displaystyle\frac{q^{\frac{49m^2-7m-14}{2}}}{1-q^{49m-14}}\Bigg].
\end{multline}
Comparing the coefficients of $q,q^2,q^3,q^4,q^5,q^6,q^7$ in (36) and the right side of (31), and replacing $q$ by $q^{1/7},$ we get the seven equalities in Theorem 5.2.
\end{proof}
\begin{theorem}
Equations (26), (28) and (29) in Theorem 5.2 are true.
\end{theorem}
\begin{lemma}
For $i=1,2,3,$ $\Bigg(\displaystyle\sum_{n}\displaystyle\frac{(-1)^nq^{7n(n+1)/2}}{1-q^{7n+i}}\Bigg)(1-q^i)=\displaystyle\frac{(q^7;q^7)_{\infty}^2}{(q^{7-i},q^{7+i};q^7)_{\infty}}$
\end{lemma}
\begin{proof}
Recall by Theorem 3.2 in [4] that
\begin{equation*}
    6\phi_5\Bigg[\begin{array}{c}
             z, q\sqrt z,-q\sqrt z,a_1,a_2,a_3 \\
             \sqrt z,-\sqrt z,\displaystyle\frac{zq}{a_1},\displaystyle\frac{zq}{a_2},\displaystyle\frac{zq}{a_3}
    \end{array}
;q,\displaystyle\frac{zq}{a_1a_2a_3}\Bigg]=\displaystyle\prod\Bigg[\begin{array}{c}
             zq, \displaystyle\frac{zq}{a_1a_2},\displaystyle\frac{zq}{a_1a_3},\displaystyle\frac{zq}{a_2a_3} \\
             \displaystyle\frac{zq}{a_1},\displaystyle\frac{zq}{a_2},\displaystyle\frac{zq}{a_3},\displaystyle\frac{zq}{a_1a_2a_3}
    \end{array}
    \Bigg]
\end{equation*}
Setting $z=1,a_1\rightarrow\infty,a_2=q^i,a_3=q^{-i},q=q^7,$ the right side equals the right side of Lemma 5.5. Under the same substitutions, the left side becomes 
\begin{multline*}
    1+\displaystyle\sum_{n=1}^{\infty}(-1)^nq^{7n(n+1)/2}(1+q^{7n})\Bigg(\displaystyle\frac{(1-q^i)(1-q^{-i})}{(1-q^{7n+i})(1-q^{7n-i})}\Bigg)\\
    =1+\displaystyle\sum_{n=1}^{\infty}(-1)^nq^{7n(n+1)/2}(1-q^i)(1-q^{-i})\Bigg(\displaystyle\frac{1}{(1-q^i)(1-q^{7n-i})}+\displaystyle\frac{1}{(1-q^{-i})(1-q^{7n+i})}\Bigg) \\ 
    =1+\displaystyle\sum_{n\geq 1}\displaystyle\frac{(-1)^nq^{7n(n+1)/2}(1-q^i)}{1-q^{7n+i}}+\displaystyle\sum_{n\leq -1}\displaystyle\frac{(-1)^nq^{7n(n+1)/2}(1-q^i)}{1-q^{7n+i}}=(1-q^i)\displaystyle\sum_{n}\displaystyle\frac{(-1)^nq^{7n(n+1)/2}}{1-q^{7n+i}},
\end{multline*}
proving the lemma.
\end{proof}
Setting $i=1,2,3$ in Lemma 5.5, we get the following.
\begin{equation}
    \displaystyle\sum_{n}\displaystyle\frac{(-1)^nq^{7n(n+1)/2}}{1-q^{7n+1}}=(q^7;q^7)_{\infty}^2L(q)
\end{equation}
\begin{equation}
    \displaystyle\sum_{n}\displaystyle\frac{(-1)^nq^{7n(n+1)/2}}{1-q^{7n+2}}=(q^7;q^7)_{\infty}^2N(q)
\end{equation}
\begin{equation}
    \displaystyle\sum_{n}\displaystyle\frac{(-1)^nq^{7n(n+1)/2}}{1-q^{7n+3}}=(q^7;q^7)_{\infty}^2Q(q)
\end{equation}
\begin{proof}[Proof of Theorem 5.4]
We have that $m(q^3,q^7,q^{-1})=-\displaystyle\frac{1}{qj(1/q;q^7)}\displaystyle\sum_{m}\displaystyle\frac{(-1)^mq^{(7m^2+5m)/2}}{1-q^{7m+2}}$ and \\ $m(q^3,q^7,q^{-6})=-\displaystyle\frac{1}{q^6j(1/q^6;q^7)}\displaystyle\sum_{m}\displaystyle\frac{(-1)^mq^{(7m^2-5m)/2}}{1-q^{7m-3}}$. It thereby follows that\\
 
$\displaystyle\sum_{m}\displaystyle\frac{(-1)^mq^{(7m^2+5m)/2}}{1-q^{7m+2}}-\displaystyle\sum_{m}\displaystyle\frac{(-1)^mq^{(7m^2-5m)/2}}{1-q^{7m-3}}\\ \\
=-qj(1/q;q^7)m(q^3,q^7,q^{-1})+q^6j(1/q^6;q^7)m(q^3,q^7,q^{-6})
=(q,q^6,q^7;q^7)_{\infty}\Bigg[m(q^3,q^7,q^{-1})-m(q^3,q^7,q^{-6})\Bigg]$

Using Lemma 11.3.4 of [5], we can check as before that\\ \\ $m(q^3,q^7,q^{-1})-m(q^3,q^7,q^{-6})=\displaystyle\frac{\displaystyle\frac{1}{q^6}(q^7;q^7)_{\infty}^3j(q^5;q^7)j(1/q^4;q^7)}{j(1/q^6;q^7)j(1/q;q^7)j(1/q^3;q^7)j(q^2;q^7)}=\displaystyle\frac{(q^7;q^7)_{\infty}}{(q,q^6,q,q^6;q^7)_{\infty}}$ \\ \\
so that $\displaystyle\sum_{m}\displaystyle\frac{(-1)^mq^{(7m^2+5m)/2}}{1-q^{7m+2}}-\displaystyle\sum_{m}\displaystyle\frac{(-1)^mq^{(7m^2-5m)/2}}{1-q^{7m-3}}=\displaystyle\frac{(q^7;q^7)_{\infty}^2}{(q,q^6;q^7)_{\infty}}=(q^7;q^7)_{\infty}^2L(q)$ \\ \\whence by (37), equation (26) follows.\\ \\
The proof of equation (28) is similar. Indeed, $m(q,q^7,q^{-2})=-\displaystyle\frac{1}{q^2j(1/q^2;q^7)}\displaystyle\sum_{m}\displaystyle\frac{(-1)^mq^{(7m^2+3m)/2}}{1-q^{7m-1}}$ and $m(q,q^7,q^{2})=-\displaystyle\frac{q^2}{j(q^2;q^7)}\displaystyle\sum_{m}\displaystyle\frac{(-1)^mq^{(7m^2+11m)/2}}{1-q^{7m+3}}$ whence similarly as before, we can check by Lemma 11.3.4 of [5] that \\ \\$\displaystyle\sum_{m}\displaystyle\frac{(-1)^mq^{(7m^2+3m-2)/2}}{1-q^{7m-1}}+\displaystyle\sum_{m}\displaystyle\frac{(-1)^mq^{(7m^2+11m+2)/2}}{1-q^{7m+3}}=-(q^7;q^7)_{\infty}^2N(q).$ \\ \\By (38), equation (28) now follows. \\ \\ Equation (29) follows similarly from (39). We omit the details.
\end{proof}
\begin{theorem}
Equations (24), (25) and (27) in the statement of Theorem 5.2 are true.
\end{theorem}
First we prove the following lemma.
\begin{lemma}
\begin{equation*}
    (I) -\displaystyle\frac{L^2(q)}{N(q)}+q\displaystyle\frac{Q(q)N(q)}{L(q)}=-\displaystyle\frac{L(q)N^2(q)}{Q^2(q)} 
\end{equation*}
\begin{equation*}
    (II) \displaystyle\frac{L(q)N(q)}{Q(q)}+q\displaystyle\frac{Q^2(q)}{L(q)}=\displaystyle\frac{L^2(q)Q(q)}{N^2(q)} 
\end{equation*}
\begin{equation*}
    (III) \displaystyle\frac{N^2(q)}{Q(q)}-\displaystyle\frac{L(q)Q(q)}{N(q)}=-q\displaystyle\frac{Q^2(q)N(q)}{L^2(q)}
\end{equation*}
\end{lemma}
\begin{proof}
(I) is equivalent to \\ \\$-\displaystyle\frac{(q^2,q^5;q^7)_{\infty}}{(q,q^6,q,q^6;q^7)_{\infty}}+q\displaystyle\frac{(q,q^6;q^7)_{\infty}}{(q^2,q^3,q^4,q^5;q^7)_{\infty}}+\displaystyle\frac{(q^3,q^4,q^3,q^4;q^7)_{\infty}}{(q,q^6,q^2,q^5,q^2,q^5;q^7)_{\infty}}=0$ 
\begin{multline}
\Longleftrightarrow\displaystyle\frac{1}{(q;q)_{\infty}^2}\Bigg[-(q^2,q^5,q^2,q^5,q^2,q^5,q^3,q^4,q^3,q^4,q^7,q^7;q^7)_{\infty}\\ \\+q(q,q^6,q,q^6,q,q^6,q^2,q^5,q^3,q^4,q^7,q^7;q^7)_{\infty}+(q^3,q^4,q^3,q^4,q^3,q^4,q^3,q^4,q,q^6,q^7,q^7;q^7)_{\infty}\Bigg]=0
\end{multline}
which is proved by $-(q^2,q^5,q^2,q^5,q^2,q^5,q^3,q^4;q^7)_{\infty}\\ \\+q(q,q^6,q,q^6,q,q^6,q^2,q^5;q^7)_{\infty}+(q^3,q^4,q^3,q^4,q^3,q^4,q,q^6;q^7)_{\infty}=0.$ \\ \\
Similarly (II) is equivalent to 
\begin{multline}
\displaystyle\frac{1}{(q;q)_{\infty}^2}\Bigg[(q^3,q^4,q^3,q^4,q^3,q^4,q,q^6,q^2,q^5,q^7,q^7;q^7)_{\infty}\\ \\+q(q,q^6,q,q^6,q,q^6,q^2,q^5,q^2,q^5,q^7,q^7;q^7)_{\infty}-(q^2,q^5,q^2,q^5,q^2,q^5,q^2,q^5,q^3,q^4,q^7,q^7;q^7)_{\infty}\Bigg]=0
\end{multline}
while (III) is equivalent to 
\begin{multline}
\displaystyle\frac{1}{(q;q)_{\infty}^2}\Bigg[(q^3,q^4,q^3,q^4,q^3,q^4,q,q^6,q,q^6,q^7,q^7;q^7)_{\infty}\\ \\-(q^2,q^5,q^2,q^5,q^2,q^5,q,q^6,q^3,q^4,q^7,q^7;q^7)_{\infty}+q(q,q^6,q,q^6,q,q^6,q,q^6,q^2,q^5,q^7,q^7;q^7)_{\infty}\Bigg]=0
\end{multline}
Dividing (40), (41) and (42) by $(q^3,q^4,q^7,q^7;q^7)_{\infty},(q^2,q^5,q^7,q^7;q^7)_{\infty}$ and $(q,q^6,q^7,q^7;q^7)_{\infty}$ respectively, we see that all three equations in Lemma 5.7 are proved by 
\begin{equation}
    (q^3,q^4,q^3,q^4,q^3,q^4,q,q^6;q^7)_{\infty}-(q^2,q^5,q^2,q^5,q^2,q^5,q^3,q^4;q^7)_{\infty}+q(q,q^6,q,q^6,q,q^6,q^2,q^5;q^7)_{\infty}=0
\end{equation}
Lemma 7.4.4 in [6] is 
\begin{multline*}
    (\displaystyle\frac{aq}{b},\displaystyle\frac{b}{a},\displaystyle\frac{aq}{ef},\displaystyle\frac{ef}{a},\displaystyle\frac{aq}{df},\displaystyle\frac{df}{a},\displaystyle\frac{aq}{bde},\displaystyle\frac{bde}{a};q)_{\infty}-(\displaystyle\frac{aq}{f},\displaystyle\frac{f}{a},\displaystyle\frac{aq}{be},\displaystyle\frac{be}{a},\displaystyle\frac{aq}{bd},\displaystyle\frac{bd}{a},\displaystyle\frac{aq}{def},\displaystyle\frac{def}{a};q)_{\infty}\\ \\
    +\displaystyle\frac{b}{a}(d,\displaystyle\frac{q}{d},e,\displaystyle\frac{q}{e},\displaystyle\frac{bq}{f},\displaystyle\frac{f}{b},\displaystyle\frac{a^2q}{bdef},\displaystyle\frac{bdef}{a^2};q)_{\infty}=0
\end{multline*}
Setting $\displaystyle\frac{b}{a}=q,\displaystyle\frac{f}{a}=q^2,e=q^2,q=q^7$ so that $\displaystyle\frac{b}{f}=\displaystyle\frac{1}{q},\displaystyle\frac{bdef}{a^2}=q^6,$ (43) and hence Lemma 5.7 are proved.
\end{proof}
\begin{proof}[Proof of Theorem 5.6]
We have $m(q^3,q^7,q^{-2})=-\displaystyle\frac{1}{q^2j(1/q^2;q^7)}\displaystyle\sum_{m}\displaystyle\frac{(-1)^mq^{(7m^2+3m)/2}}{1-q^{7m+1}}$ and $m(q^3,q^7,q^{-5})=-\displaystyle\frac{1}{q^5j(1/q^5;q^7)}\displaystyle\sum_{m}\displaystyle\frac{(-1)^mq^{(7m^2-3m)/2}}{1-q^{7m-2}}$ so that the left side of (24) equals \\ \\
$q^2j(1/q^2;q^7)m(q^3,q^7,q^{-2})-q^5j(1/q^5;q^7)m(q^3,q^7,q^{-5})=(q^2,q^5,q^7;q^7)_{\infty}\Bigg[m(q^3,q^7,q^{-5})-m(q^3,q^7,q^{-2})\Bigg].$ \\ \\
By Lemma 11.3.4 of [5], we check as before that 
\begin{align*}
    m(q^3,q^7,q^{-5})-m(q^3,q^7,q^{-2})=\displaystyle\frac{\displaystyle\frac{1}{q^2}(q^7;q^7)_{\infty}^3j(1/q^3;q^7)j(1/q^4;q^7)}{j(1/q^2;q^7)j(1/q^5;q^7)j(q;q^7)j(1/q^2;q^7)}=-\displaystyle\frac{(q^3,q^3,q^4,q^4,q^7;q^7)_{\infty}}{(q,q^6,q^2,q^5,q^2,q^5,q^2,q^5;q^7)_{\infty}}
\end{align*}
so that the left side of (24) equals $-(q^7;q^7)_{\infty}^2\displaystyle\frac{L(q)N^2(q)}{Q^2(q)}$ whence (24) follows from (I) of Lemma 5.7. \\ \\
For (25), we start with $m(q^2,q^7,q^{-1})=-\displaystyle\frac{1}{qj(1/q;q^7)}\displaystyle\sum_{m}\displaystyle\frac{(-1)^mq^{(7m^2+5m)/2}}{1-q^{7m+1}}$ and \\ \\ $m(q^2,q^7,q)=-\displaystyle\frac{q}{j(q;q^7)}\displaystyle\sum_{m}\displaystyle\frac{(-1)^mq^{(7m^2+9m)/2}}{1-q^{7m+3}}$ so that the left side of (25) equals  \\ \\
$4(q,q^6,q^7;q^7)_{\infty}\Bigg[m(q^2,q^7,q^{-1})-m(q^2,q^7,q)\Bigg].$ \\ \\
By Lemma 11.3.4 of [5],$m(q^2,q^7,q^{-1})-m(q^2,q^7,q)=\displaystyle\frac{(q^2,q^5,q^2,q^5,q^7;q^7)_{\infty}}{(q,q^6,q,q^6,q,q^6,q^3,q^4;q^7)_{\infty}}$
so that the left side of (25) equals $4(q^7;q^7)_{\infty}^2\displaystyle\frac{L^2(q)Q(q)}{N^2(q)}$ whence (25) follows from (II) of Lemma 5.7. \\ \\
Equation (27) follows in exactly the same way from (III) of Lemma 5.7. We again omit the details. 

\end{proof} 
The only equation left to be proven to finish the proof of Conjecture 5.1 and hence of Theorem 1.4 is equation (30) in the statement of Theorem 5.2. We haven't been able to show this, but the following theorem gives different expressions for the left side of this equation. 
\begin{theorem}
If (D) is the statement
\begin{multline*}
    \displaystyle\sum_{m}'(-1)^m\displaystyle\frac{q^{\frac{7m^2+m-2}{2}}}{1-q^{7m}}-\displaystyle\sum_{m}(-1)^m\displaystyle\frac{q^{\frac{7m^2-m-2}{2}}}{1-q^{7m-1}}-4\displaystyle\sum_{m}'(-1)^m\displaystyle\frac{q^{\frac{7m^2+3m-2}{2}}}{1-q^{7m}}\\ \\
    +4\displaystyle\sum_{m}(-1)^m\displaystyle\frac{q^{\frac{7m^2-3m-2}{2}}}{1-q^{7m-3}}+5\displaystyle\sum_{m}'(-1)^m\displaystyle\frac{q^{\frac{7m^2+5m-2}{2}}}{1-q^{7m}}+5\displaystyle\sum_{m}(-1)^m\displaystyle\frac{q^{\frac{7m^2+9m}{2}}}{1-q^{7m+2}}\\ \\
    =3(q^7;q^7)_{\infty}^2\Bigg(\displaystyle\frac{Q^2(q)}{N(q)}+\displaystyle\frac{N^2(q)}{L(q)}\Bigg),
\end{multline*}
(E) is the statement 
\begin{multline*}
    \displaystyle\frac{3}{q^2}\displaystyle\sum_{j=0}^{\infty}\Bigg(\displaystyle\frac{q^{7j+5}}{1-q^{7j+4}}-\displaystyle\frac{q^{7j+4}}{1-q^{7j+3}}+2\displaystyle\frac{q^{7j+2}}{1-q^{7j+1}}-2\displaystyle\frac{q^{7j+7}}{1-q^{7j+6}}+3\displaystyle\frac{q^{7j+6}}{1-q^{7j+5}}-3\displaystyle\frac{q^{7j+3}}{1-q^{7j+2}}\Bigg)\\ \\
    =3(q^7;q^7)_{\infty}^2\Bigg(\displaystyle\frac{Q^2(q)}{N(q)}+\displaystyle\frac{N^2(q)}{L(q)}\Bigg)
\end{multline*}
and (F) is the statement
\begin{equation*}
    3\displaystyle\sum_{m}\displaystyle\frac{q^m(1-q^{m+1})^3(2q^{2m+2}+3q^{m+1}+2)}{1-q^{7m+7}}=3(q^7;q^7)_{\infty}^2\Bigg(\displaystyle\frac{Q^2(q)}{N(q)}+\displaystyle\frac{N^2(q)}{L(q)}\Bigg),
\end{equation*}
then (D), (E) and (F) are equivalent.
\end{theorem}
\begin{proof}
We reduce the left side of (D) to the left side of (E) and the left side of (E) to the left side of (F). \\ \\
Let $U_1=\displaystyle\sum_{m}'(-1)^m\displaystyle\frac{q^{\frac{7m^2+m-2}{2}}}{1-q^{7m}}, U_2=\displaystyle\sum_{m}(-1)^m\displaystyle\frac{q^{\frac{7m^2-m-2}{2}}}{1-q^{7m-1}}, U_3=\displaystyle\sum_{m}'(-1)^m\displaystyle\frac{q^{\frac{7m^2+3m-2}{2}}}{1-q^{7m}},\\ \\U_4=\displaystyle\sum_{m}(-1)^m\displaystyle\frac{q^{\frac{7m^2-3m-2}{2}}}{1-q^{7m-3}},U_5=\displaystyle\sum_{m}'(-1)^m\displaystyle\frac{q^{\frac{7m^2+5m-2}{2}}}{1-q^{7m}},U_6=\displaystyle\sum_{m}(-1)^m\displaystyle\frac{q^{\frac{7m^2+9m}{2}}}{1-q^{7m+2}},$ \\ \\ and let $\phi=U_1-U_2,\psi=U_3-U_4,\rho=U_5+U_6.$ \\ \\
The left side of (D) is thus $U_1-U_2-4U_3+4U_4+5U_5+5U_6=\phi-4\psi+5\rho.$ \\ \\
Since $m(q^3,q^7,1/q^4)=-\displaystyle\frac{1}{q^4j(1/q^4;q^7)}\displaystyle\sum_{m}\displaystyle\frac{(-1)^mq^{(7m^2-m)/2}}{1-q^{7m-1}},$ we get $U_2=\displaystyle\frac{1}{q}(q^3,q^4,q^7;q^7)_{\infty}m(q^3,q^7,1/q^4).$ Similarly, $U_4=\displaystyle\frac{1}{q}(q^2,q^5,q^7;q^7)_{\infty}m(q^2,q^7,1/q^5)$ and $U_6=-\displaystyle\frac{1}{q}(q,q^6,q^7;q^7)_{\infty}m(q,q^7,q).$ \\ \\
If $m'(x,q,z)=-\displaystyle\frac{z}{j(z;q)}\displaystyle\sum_{r}'\displaystyle\frac{(-1)^rq^{r(r+1)/2}}{1-xzq^r}$, then $m'(q^3,q^7,1/q^3)=-\displaystyle\frac{1}{q^3j(1/q^3;q^7)}\displaystyle\sum_{m}'\displaystyle\frac{(-1)^mq^{(7m^2+m)/2}}{1-q^{7m}}$ so that \\ \\ $U_1=-q^2j(1/q^3;q^7)m'(q^3,q^7,1/q^3)=\displaystyle\frac{1}{q}(q^3,q^4,q^7;q^7)_{\infty}\lim_{z\to q^{-3}}\Bigg(m(q^3,q^7,z)+\displaystyle\frac{z}{j(z;q^7)}.\displaystyle\frac{1}{1-q^3z}\Bigg).$ \\ \\
Similarly, $U_3=\displaystyle\frac{1}{q}(q^2,q^5,q^7;q^7)_{\infty}\lim_{z\to q^{-2}}\Bigg(m(q^2,q^7,z)+\displaystyle\frac{z}{j(z;q^7)}.\displaystyle\frac{1}{1-q^2z}\Bigg)$ and \\ \\
$U_5=\displaystyle\frac{1}{q}(q,q^6,q^7;q^7)_{\infty}\lim_{z\to q^{-1}}\Bigg(m(q,q^7,z)+\displaystyle\frac{z}{j(z;q^7)}.\displaystyle\frac{1}{1-qz}\Bigg).$ \\ \\
Now, $\phi=U_1-U_2=-\displaystyle\frac{1}{q}(q^3,q^4,q^7;q^7)_{\infty}\lim_{z\to q^{-3}}\Bigg(m(q^3,q^7,q^{-4})-m(q^3,q^7,z)-\displaystyle\frac{z}{j(z;q^7)}.\displaystyle\frac{1}{1-q^3z}\Bigg).$ \\ \\
By Lemma 11.3.4 in [5], we verify that 
\begin{multline*}
    m(q^3,q^7,q^{-4})-m(q^3,q^7,z)=\displaystyle\frac{z(q^7;q^7)_{\infty}^3j(1/q^4z;q^7)j(z/q;q^7)}{j(z;q^7)j(1/q^4;q^7)j(q^3z;q^7)j(1/q;q^7)}\\ \\
    =\displaystyle\frac{z(q^7,\displaystyle\frac{1}{q^4z},q^{11}z,\displaystyle\frac{z}{q},\displaystyle\frac{q^8}{z};q^7)_{\infty}}{(z,\displaystyle\frac{q^7}{z},\displaystyle\frac{1}{q^4},q^{11},q^3z,\displaystyle\frac{q^4}{z},\displaystyle\frac{1}{q},q^8;q^7)_{\infty}}
\end{multline*}
so that 
\begin{multline*}
    \phi=-\displaystyle\frac{1}{q}(q^3,q^4,q^7;q^7)_{\infty}\lim_{z\to q^{-3}}\Bigg[\displaystyle\frac{z(q^7,\displaystyle\frac{1}{q^4z},q^{11}z,\displaystyle\frac{z}{q},\displaystyle\frac{q^8}{z};q^7)_{\infty}}{(z,\displaystyle\frac{q^7}{z},\displaystyle\frac{1}{q^4},q^{11},q^3z,\displaystyle\frac{q^4}{z},\displaystyle\frac{1}{q},q^8;q^7)_{\infty}}-\displaystyle\frac{z}{j(z;q^7)}.\displaystyle\frac{1}{1-q^3z}\Bigg]
\\ \\
    =-\displaystyle\frac{1}{q}(q^3,q^4,q^7;q^7)_{\infty}\lim_{z\to q^{-3}}\Bigg[\displaystyle\frac{z}{(z,\displaystyle\frac{q^7}{z},\displaystyle\frac{1}{q^4},q^{11},q^{10}z,\displaystyle\frac{q^4}{z},\displaystyle\frac{1}{q},q^8,q^7;q^7)_{\infty}} 
\\ \\
    \times\displaystyle\frac{(q^7,q^7,\displaystyle\frac{1}{q^4z},q^{11}z,\displaystyle\frac{z}{q},\displaystyle\frac{q^8}{z};q^7)_{\infty}-(\displaystyle\frac{1}{q^4},q^{11},q^{10}z,\displaystyle\frac{q^4}{z},\displaystyle\frac{1}{q},q^8;q^7)_{\infty}}{(1-q^3z)}\Bigg]
\end{multline*}
\begin{multline*}
    =-\displaystyle\frac{1}{q}(q^3,q^4,q^7;q^7)_{\infty}\displaystyle\frac{\displaystyle\frac{1}{q^3}}{(\displaystyle\frac{1}{q^3},q^{10},\displaystyle\frac{1}{q^4},q^{11},q^7,q^7,\displaystyle\frac{1}{q},q^8,q^7;q^7)_{\infty}} \\ \\
    \times\lim_{z\to q^{-1}}\displaystyle\frac{(q^7,q^7,\displaystyle\frac{1}{q^2z},q^{9}z,\displaystyle\frac{z}{q^3},\displaystyle\frac{q^{10}}{z};q^7)_{\infty}-(\displaystyle\frac{1}{q^4},q^{11},q^{8}z,\displaystyle\frac{q^6}{z},\displaystyle\frac{1}{q},q^8;q^7)_{\infty}}{(1-qz)}
\end{multline*}
\begin{multline*}
    =\displaystyle\frac{q^4}{(q^3,q^4,q^7,q,q^6,q^7;q^7)_{\infty}}\lim_{z\to q^{-1}}\displaystyle\frac{(q^7,q^7,\displaystyle\frac{1}{q^2z},q^{9}z,\displaystyle\frac{z}{q^3},\displaystyle\frac{q^{10}}{z};q^7)_{\infty}-(\displaystyle\frac{1}{q^4},q^{11},q^{8}z,\displaystyle\frac{q^6}{z},\displaystyle\frac{1}{q},q^8;q^7)_{\infty}}{(1-qz)}
\end{multline*}
Now, if $F_1(z)=(q^7,q^7,\displaystyle\frac{1}{q^2z},q^{9}z,\displaystyle\frac{z}{q^3},\displaystyle\frac{q^{10}}{z};q^7)_{\infty},$ then we can check that by differentiation, \\ \\
$F_1'(\displaystyle\frac{1}{q})=\displaystyle\frac{(q^7;q^7)_{\infty}^2(q,q^6,q^3,q^4;q^7)_{\infty}}{q^5}\sum_{j=0}^{\infty}\Bigg(\displaystyle\frac{q^{7j}}{1-q^{7j-1}}-\displaystyle\frac{q^{7j+9}}{1-q^{7j+8}}-\displaystyle\frac{q^{7j-3}}{1-q^{7j-4}}+\displaystyle\frac{q^{7j+12}}{1-q^{7j+11}}\Bigg).$ \\ \\
Similarly, if $F_2(z)=(\displaystyle\frac{1}{q^4},q^{11},q^{8}z,\displaystyle\frac{q^6}{z},\displaystyle\frac{1}{q},q^8;q^7)_{\infty},$ then $F_2'(\displaystyle\frac{1}{q})=0.$ \\ \\
This implies that $\phi=\displaystyle\frac{1}{q^2}\sum_{j=0}^{\infty}\Bigg(-\displaystyle\frac{q^{7j}}{1-q^{7j-1}}+\displaystyle\frac{q^{7j+9}}{1-q^{7j+8}}+\displaystyle\frac{q^{7j-3}}{1-q^{7j-4}}-\displaystyle\frac{q^{7j+12}}{1-q^{7j+11}}\Bigg)$ by L'Hospitals rule.\\ \\
In an exactly similar fashion, we get \\ \\ $\psi=\displaystyle\frac{1}{q^2}\sum_{j=0}^{\infty}\Bigg(-\displaystyle\frac{q^{7j-2}}{1-q^{7j-3}}+\displaystyle\frac{q^{7j+11}}{1-q^{7j+10}}+\displaystyle\frac{q^{7j-4}}{1-q^{7j-5}}-\displaystyle\frac{q^{7j+13}}{1-q^{7j+12}}\Bigg)$ \\ \\
and $\rho=-\displaystyle\frac{1}{q^2}\sum_{j=0}^{\infty}\Bigg(\displaystyle\frac{q^{7j+3}}{1-q^{7j+2}}-\displaystyle\frac{q^{7j+6}}{1-q^{7j+5}}-\displaystyle\frac{q^{7j+2}}{1-q^{7j+1}}+\displaystyle\frac{q^{7j+7}}{1-q^{7j+6}}\Bigg).$ \\ \\
Collecting like powers of $q$ modulo 7, this gives after simplification that the left side of (D) equals the left side of (E). \\ \\
Further, noting that $\displaystyle\sum_{j\geq 0}\displaystyle\frac{q^{7j}}{1-q^{7j+A}}=\displaystyle\sum_{j,m\geq 0}q^{7j+7jm+7Am}=\displaystyle\sum_{m\geq 0}\displaystyle\frac{q^{Am}}{1-q^{7m+7}},$ \\ \\
the left side of (E) reduces to 
\begin{multline*}
    \displaystyle\frac{3}{q^2}\displaystyle\sum_{m\geq 0}\displaystyle\frac{q^{4m+5}-q^{3m+4}+2q^{m+2}-2q^{6m+7}+3q^{5m+6}-3q^{2m+3}}{1-q^{7m+7}} \\ \\
    =3\displaystyle\sum_{m\geq 0}\displaystyle\frac{q^m(1-q^{m+1})^3(2q^{2m+2}+3q^{m+1}+2)}{1-q^{7m+7}}
\end{multline*}
on factorization, which is the left side of (F). The proof of theorem 5.8 is now complete.

\end{proof}
\section{Conclusion and Further Work}
In this paper, we have thus given a new approach to proving George-Beck's conjectures on the crank of a partition modulo 5 or 7. The Apple-Lerch sums we come across our decompositions are intricately related to tenth-order mock theta functions as discussed in [5] Chapter 11. This points to the possibility that George Beck's conjectures, which appear striking and unique at first glance, may actually be coming from identities in tenth order mock-theta functions. We note however, that by our approach, we have not managed to prove completely the crank conjectures, missing an identity each in the cases modulo 5 and 7. These unproven identities seem to be a little different that the ones we prove in this paper, and we cannot rule out the possibility that they are actually special cases of a wider range of results, again intricately related to mock-theta functions.

\end{document}